\documentclass[11pt]{article}
\usepackage{anysize}
\usepackage[utf8]{inputenc}
\usepackage[T1]{fontenc}
\usepackage{amsmath, amssymb, amsthm, latexsym, color, textcomp, mathtools}
\usepackage{paralist}
\usepackage[T1]{fontenc}
\usepackage{lmodern}
\usepackage{url}
\usepackage{tikz}
\usetikzlibrary{arrows,calc,shapes} 

\newcommand{\R}{\mathbb{R}}
\newcommand{\Z}{\mathbb{Z}}  
\newcommand\hz[1]{#1\hphantom{0}}

\newcommand\HM{\hphantom{-}}
\newcommand{\fset}{{\mathcal{F}}}

\newcommand{\size}{\operatorname{size}} 
\newcommand{\cone}{\operatorname{cone}}
 
\theoremstyle{plain}
\newtheorem{theorem}{Theorem}[section]
\newtheorem{proposition}[theorem]{Proposition}

\theoremstyle{definition}
\newtheorem{alg}[theorem]{Algorithm}
\newtheorem{definition}[theorem]{Definition}
\newtheorem{example}[theorem]{Example}
\newtheorem{examples}[theorem]{Examples}%

\theoremstyle{remark}
\newtheorem{observations}[theorem]{Observations}

\usepackage{booktabs,ltxtable}
\title{Small \emph{f}-vectors of {3}-spheres and of {4}-polytopes}
\author{Philip Brinkmann%
\footnote{Funded by DFG through the RTG \emph{Methods for Discrete Structures}.}\\ 
Institut f\"ur Mathematik, FU Berlin\\Arnimallee 2\\14195 Berlin, Germany\\
\url{philip.brinkmann@fu-berlin.de}
\and
G\"unter M.~Ziegler%
\setcounter{footnote}{6}%
\footnote{Supported by DFG via the Collaborative Research Center TRR~109 ``Discretization in Geometry and Dynamics''.}\\
Institut f\"ur Mathematik, FU Berlin\\Arnimallee 2\\14195 Berlin, Germany\\
\url{ziegler@math.fu-berlin.de}}
\date{{\small October 4, 2016}}

\begin{document}
    
    \maketitle
    
\begin{abstract}
	We present a new algorithmic approach that
	can be used to determine whether a given
	quadruple $(f_0,f_1,f_2,f_3)$ is 
	the $f$-vector of any convex $4$-dimensional polytope. 
	By implementing this approach, we classify the $f$-vectors  
	of $4$-polytopes in the range $f_0+f_3\le22$. 
	
	In particular, we thus prove that there are $f$-vectors of  
	cellular $3$-spheres with the intersection property
	that are not $f$-vectors of any convex $4$-polytopes,
	thus answering a question that may be traced back to the works of
	Steinitz (1906/1922). In the
	range $f_0+f_3\le22$, there are exactly three such $f$-vectors with $f_0\le f_3$,
	namely $(10,32,33,11)$, $(10,33,35,12)$, and $(11,35,35,11)$.
\end{abstract}

\section{Introduction} 

In 1906, Ernst Steinitz \cite{Steinitz} proved a remarkably simple and complete result:
The set of all $f$-vectors of $3$-polytopes is given by
all the integer points in a $2$-dimensional polyhedral cone,
whose boundary is given by the extremal cases of ($f$-vectors of)
simple and of simplicial polytopes:
\[
\fset(\mathcal{P}^3) = \{(f_0,f_1,f_2)\in\Z^3: f_0-f_1+f_2=2,\ f_2\le2f_0-4,\ f_0\le2f_2-4\}.
\]
Steinitz's later work \cite{SteinitzThm,SteinitzVorlesung} from 1922/1934 
implies that the same characterization
is valid also for the $f$-vectors of more general objects 
such as regular cellular 2-spheres with the intersection property or
of interval-connected Eulerian lattices of length $4$ (as described below).
 
{The \emph{f}-vectors of 4-polytopes}, however, provide a much greater
challenge.
Grünbaum wrote in his 1967 book:  
\begin{compactitem}
	\item[]
	“\emph{It would be rather interesting to find a characterization of those lattice
		points in~$\R^4$ which are the $f$-vectors of $4$-polytopes. This goal seems rather
		distant, however, in view of our inability to solve even such a small part 
		of the problem as the lower bound conjecture for $4$-polytopes.}” 
		(Grünbaum \cite[p.~191]{Gruenbaum})
\end{compactitem}
The lower bound conjecture was solved by Barnette in 1971/73  \cite{Bar,Bar1}, 
but the problem to characterize $\fset(\mathcal{P}^4)$ remains wide open.
Grünbaum himself initiated and started in \cite[Sect.~10.4]{Gruenbaum} 
a study of the $2$-dimensional coordinate
projections of the $3$-dimensional set $\fset(\mathcal{P}^4)\subset\R^4$,
which was eventually completed by 
Barnette and Reay \cite{barnette73:_projec} and Barnette \cite{barnette74:_e_s}.
A typical result in the series says that a pair $(f_i,f_j)$ occurs in an 
$f$-vector if it satisfies some simple linear or quadratic upper/lower bound inequalities,
and if it is not one of finitely-many “small” exceptions.
For example, according to \cite[Thm.\,10.4.1]{Gruenbaum} 
a pair $(f_0,f_3)$ occurs for a $4$-polytope if
and only if the upper bound inequalities
$f_3\le\frac12f_0(f_0-3)$ and
$f_0\le\frac12f_3(f_3-3)$ are satisfied, with no exceptions in this case.

Any characterization of (a projection of) the set of $f$-vectors $\fset(\mathcal{P}^4)$
contains a characterization of the extremal cases
and a solution of the corresponding extremal problems.
Some of these are visible in $2$-dimensional coordinate projections.
For example, the $(f_0,f_3)$-classification quoted above contains the upper bound theorem
for $4$-polytopes. 

As a complete determination of $\fset(\mathcal{P}^4)$ seems out of reach, 
a natural approximation to the problem asks for a characterization of the
closed convex cones with
apex at the $f$-vector $f(\Delta_4)=(5,10,10,5)$ of the $4$-simplex that are generated by the $f$-vectors of $4$-polytopes 
resp.\ of $3$-spheres,
\[
  \cone(\fset(\mathcal{P}^4)) \subseteq
  \cone(\fset(\mathcal{S}^3)).
\]
Equivalently, one asks for the linear inequalities that are valid for all $f$-vectors and
tight at $f(\Delta_4)=(5,10,10,5)$. 
For example, the inequalities $f_1\ge2f_0$ and $f_2\ge2f_3$ are of this form,
satisfied with equality by simple resp.\ simplicial $4$-polytopes. 
Thus, in particular the $f$-vectors of simple and simplicial $4$-polytopes 
are extremal in the coordinate projections to $(f_0,f_1)$ resp.\ $(f_2,f_3)$.

It was noted in Ziegler \cite{Z82} that a key parameter of an $f$-vector is the \emph{fatness} 
\[
F(f_0,f_1,f_2,f_3):=\frac{f_1+f_2-20}{f_0+f_3-10}.
\]
Though fatness is not defined for the ($f$-vector of a) simplex,
every lower or upper bound on fatness 
corresponds to a linear inequality that is tight at the simplex.
In \cite{Z82} the second author also identified the
two key problems that prevent us up to now from determining 
$\cone(\fset(\mathcal{P}^4))$ or
$\cone(\fset(\mathcal{S}^3))$:
\begin{compactitem}
	\item \emph{Does fatness have an upper bound for $4$-polytopes}?\\
	(It does not for $3$-spheres, as proved by Eppstein, Kuperberg \& Ziegler \cite{Z80}.)
	\item \emph{Is the fatness lower bound $F\ge2.5$ valid for all $3$-spheres}?\\
	(For $4$-polytopes it follows from $g_2^{tor}\ge0$, see Kalai \cite{kalai87:_rigid_i}.)%
\end{compactitem}
These are extremal problems on 
$\fset(\mathcal{P}^4)$ resp.\  
$\fset(\mathcal{S}^3)$ 
that cannot be solved by looking at the projections to only two coordinates.
However, below we will suggest a different projection which displays fatness very clearly.
\smallskip
 
{In this paper} we are not directly dealing with the asymptotic questions.
Rather we classify the $f$-vectors of ``small'' polytopes, 
and from this derive new insights into what happens asymptotically.
For this, we redefine ``small'' by measuring the \emph{size} of an $f$-vector by
\[
   \size(f_0,f_1,f_2,f_3) := f_0 + f_3 - 10.
\] 
This is a linear quantity that is $\size(5,10,10,5)=0$ for the $f$-vector of the $4$-simplex.

For the classification {we have developed a new algorithmic approach},
in order to determine for any given reasonably small $(f_0,f_1,f_2,f_3)$,
whether there is a $4$-polytope with this $f$-vector.

{We have implemented the algorithm} and achieved a complete classification of the
$f$-vectors of size up to~$12$. 
That is, for every vector $(f_0,f_1,f_2,f_3)$ with $f_0+f_2\le22$ that satisfies the known
necessary conditions on $f$-vectors of $4$-polytopes, we have either constructed
a $3$-sphere or $4$-polytope with this $f$-vector, or proved that none exists.
The results are detailed in Sections~\ref{sec:classification} and \ref{sec:examples}.
As a main consequence of the enumeration, we obtain that the difference between
spheres and polytopes is so substantial that it appears even at the level of $f$-vectors:

\begin{theorem}\label{thm:f_differ}
		The set of $f$-vectors of $4$-polytopes 
		is a strict subset of the set of $f$-vectors of strongly regular $3$-spheres:
	\[
	\fset(\mathcal{P}^4)\subsetneqq \fset(\mathcal{S}^3).
	\]
	Indeed, the sets differ in exactly five such $f$-vectors of $\size(P)=f_0+f_3-10\le12$, namely
	\begin{compactitem}%
		\item of size $11$: $(10,32,33,11)$, $(11,33,32,10)$, and  			 
		\item of size $12$: $(10,33,35,12)$, $(12,35,33,10)$, 
		$(11,35,35,11)$.   
	\end{compactitem}
\end{theorem} 

For simplicial spheres, the question whether all $f$-vectors of $(d-1)$-spheres 
also occur for $d$-polytopes---in view of the $g$-Theorem for polytopes---is equivalent to the $g$-conjecture for spheres. 
The answer is known to be “yes” for $d\le5$,
but the $g$-conjecture for sphere remains open for larger~$d$.
However, already in 1971, at the end of the paper in which he introduced the $g$-conjecture, 
McMullen voiced strong doubts: 
\begin{compactitem}
	\item[] 
	“\emph{in every case in which the [$g$-]conjecture is known to be true, it also holds
		for the corresponding triangulated spheres. 
		\emph{(\ldots)}  
		However,
		there are fundamental differences between
		triangulated $(d-1)$-spheres and boundary complexes of simplicial $d$-polytopes.
		\emph{(\ldots)} 
		We should therefore, perhaps, be wary of extending the conjecture to triangulated spheres.}”
	(McMullen \cite[p.~569]{McM1})
\end{compactitem}

Our algorithm works in three steps, proceeding from combinatorial models via 
topological models to polytopes.
It starts with an enumeration of the graphs that could be
compatible with the given $f$-vector. It then looks at the possible combinatorial types of
facets, and enumerates their combinations into an entirely combinatorial model of
polytopes, namely \emph{interval-connected Eulerian lattices} of length~$5$. 
This new model will be described in Section \ref{sec:objects}, where we also prove that
every such object corresponds to a regular cell-decomposition of a closed $3$-manifold
with the intersection property (Proposition~\ref{prop:Eulerian5}).
Thus the combinatorial types of 
regular cell-decompositions of the $3$-sphere with the intersection property
(which we simply refer to as ``$3$-spheres,'') form a subset of these Eulerian lattices.
The class of combinatorial types of convex $4$-polytopes is still more restrictive,
as became clear, for example, in the revision and correction of 
Brückner's \cite{Brueckner} work by Grünbaum \& Sreedharan \cite{Polytopes8vs}: Not every
diagram, and thus not every sphere, does correspond to a convex polytope.
 
In our search range of $\size(f)\le12$, all $f$-vectors of
Eulerian lattices also appear as $f$-vectors of spheres. That is, while
\[
\#\,\{f\in \fset(\mathcal{S}^3) \setminus \fset(\mathcal{P}^4): \size(f)\le12\} = 5
\]
we have
\[
\#\,\{f\in \fset(\mathcal{L}^5) \setminus \fset(\mathcal{S}^3): \size(f)\le12\} = 0.
\]
So it may be that $\fset(\mathcal{L}^5)=\fset(\mathcal{S}^3)$, but the computations
for $\size(f)\le12$ should not be counted as strong evidence, as indeed we did
not encounter \emph{any} manifolds that are not spheres in this range.
Also, very natural higher-dimensional versions of $\fset(\mathcal{L}^5)=\fset(\mathcal{S}^3)$ turn out to be false.  
For example, simplicial $5$-manifolds with negative $g_3$ appear in the
enumerations of Lutz \cite[pp.~56-58]{Lutz-diss}. 
						
In Figure~\ref{fig:1} we evaluate our classification results by looking 
at the $f$-vector set $\fset(\mathcal{P}^4)$
in a particular projection, which is not a coordinate projection,
and which has the virtue to show size (as first coordinate)
and fatness (as ``slope + 2'') directly. 

\begin{figure}[p]\label{fig:1}
\begin{center}
\begin{tikzpicture}[scale=.9,>=stealth']
    \draw[->] (0,0) -- coordinate (x axis mid) (9,0);
    \draw[->] (0,0) -- coordinate (y axis mid)(0,16);
    \foreach \t in {2,4,6,8,10,12,14,16,18,20}{
    	\edef\x{\t}
        \draw [](4*\x mm,1pt) -- (4*\x mm,-3pt)
            node[anchor=north] {$\x$};
    }
    \foreach \t in {2,4,6,8,10,12,14,16,18,20,22,24,26,28,30,32,34,36}{
    	\edef\y{\t}
        \draw (1pt,4*\y mm) -- (-3pt,4*\y mm) node[anchor=east] {$\y$};
    }
    \node[below=.5cm, name=X] at (x axis mid) {\hspace{60mm}$\size=f_0+f_3-10$};
    \node[rotate=90,above=1cm] at (y axis mid) {$f_1+f_2-20-2\cdot\size$};

    \draw[thick] (0,0)--(8,4) node[below right] {$F\geq\frac{5}{2}$};
    \draw[thick,dashed] (0,0)--(8,16) node[above left] {$F=4$};
    \draw[very thick,dotted] (0.4*12.5,0)--(0.4*12.5,16) node[above left] {$\size\leq 12$};
	    
    \foreach \x/\y/\beschriftung in {0/0/{}, 2/2/{}, 3/3/{}, 4/2/{}, 4/4/{}, 4/6/{},
    								 5/3/{}, 5/5/{}, 5/7/{}, 6/4/{}, 6/6/{}, 6/8/{}, 6/10/{},
    								 7/5/{}, 7/7/{}, 7/9/{}, 7/11/{}, 8/4/{}, 8/6/{}, 8/8/{},8/10/{},8/12/{},8/14/{},
    								 9/5/{}, 9/7/{}, 9/9/{}, 9/11/{}, 9/13/{}, 9/15/{},
    								 10/6/{}, 10/8/{}, 10/8/{}, 10/10/{}, 10/12/{}, 10/12/{}, 10/14/{}, 10/16/{},
    								 10/20/{$\Delta_2(4)$},
    								 11/7/{}, 11/9/{}, 11/11/{}, 11/13/{}, 11/15/{}, 11/17/{},
    								 12/6/{}, 12/8/{}, 12/10/{}, 12/12/{}, 12/14/{}, 12/16/{}, 12/18/{},
    								 13/7/{}, 13/9/{}, 13/11/{}, 13/13/{}, 13/15/{}, 13/17/{}, 13/19/{}, 13/21/{},
    								 14/8/{}, 14/10/{}, 14/12/{}, 14/14/{}, 14/16/{}, 14/18/{}, 14/20/{}, 14/22/{},
    								 14/24/{}, 14/26/{}, 14/28/{},
    								 15/9/{},15/11/{},15/13/{},15/15/{},15/17/{},15/19/{},15/21/{},15/23/{},15/25/{},
    								 16/8/{}, 16/10/{}, 16/12/{}, 16/14/{}, 16/16/{}, 16/18/{}, 16/20/{}, 16/22/{},
    								 16/24/{}, 16/26/{}, 16/30/{},
    								 17/9/{}, 17/11/{}, 17/13/{}, 17/15/{}, 17/17/{}, 17/19/{}, 17/21/{}, 17/23/{},
    								 17/25/{}, 17/27/{}, 17/29/{}, 17/31/{},
    								 18/10/{}, 18/12/{}, 18/14/{}, 18/16/{}, 18/18/{}, 18/20/{}, 18/22/{}, 18/24/{},
    								 18/26/{}, 18/28/{}, 18/30/{}, 18/32/{},
    								 19/11/{}, 19/13/{}, 19/15/{}, 19/17/{}, 19/19/{}, 19/21/{}, 19/23/{}, 19/25/{},
    								 19/27/{}, 19/29/{}, 19/31/{}, 19/33/{}, 19/35/{},
    								 20/10/{}, 20/12/{}, 20/14/{}, 20/16/{}, 20/18/{}, 20/20/{}, 20/22/{}, 20/24/{},
    								 20/26/{}, 20/28/{}, 20/30/{}, 20/32/{}, 20/34/{}, 20/36/{}
    								}{
        \draw (0.4*\x,0.4*\y)node[circle,radius=0.100, draw, fill=black, label={left:\beschriftung}] {};
    }
    \foreach \x/\y/\beschriftung in {11/23/{$(10,32,33,11)$}, 
										12/26/{$(11,35,35,11)$},
									 	14/32/{$(12,40,40,12)$\hspace{3mm}}
    								}{
        \draw (0.4*\x,0.4*\y)node[circle, inner sep=1.2mm, draw=black, fill=red, label={left:\beschriftung}] {};
    }
    \foreach \x/\y/\beschriftung in {10/18/{}, 11/19/{}, 11/21/{}, 12/20/{}, 12/22/{}
    								}{
        \draw (0.4*\x,0.4*\y)node[circle, inner sep=1.2mm, draw=black, fill=black!25, label={left:\beschriftung}] {};
    }
    \foreach \x/\y/\beschriftung in {8/16/{$W_9$}, 12/24/{$P_{11}$}, 16/32/{}, 18/36/{},
									14/30/{$W_{12}^{39}$\hspace*{3mm}}
    								}{
        \draw (0.4*\x,0.4*\y)node[circle, inner sep=1.2mm, draw, fill=black!25] {};
        \draw (0.4*\x,0.4*\y)node[circle, inner sep=1.2mm, cross out, draw, label={left:\beschriftung}] {};
    }
    \foreach \x/\y/\beschriftung in {13/23/{}, 13/25/{}, 13/27/{}, 13/29/{}, 13/31/{}, 13/33/{}, 14/34/{},
    								 14/36/{}, 14/38/{}, 15/27/{}, 15/29/{}, 15/31/{}, 15/33/{},
    								 15/35/{}, 15/37/{}, 15/39/{}, 16/28/{}, 16/34/{}, 16/36/{},
    								 16/38/{}, 17/33/{}, 17/35/{}, 17/37/{}, 17/39/{}, 18/34/{}, 18/38/{},
    								 19/37/{}, 19/39/{}, 20/38/{}
    								}{
        \draw (0.4*\x,0.4*\y)node[circle, inner sep=1.2mm, draw=black, fill=white, label={left:\beschriftung}] {};
    }

\end{tikzpicture} 
\end{center}
 
	This figure presents a particular $2$-dimensional projection
	of $\fset(\mathcal{P}^4)\subset\fset(\mathcal{S}^3)\subset\Z^4$:
	The $x$-axis represents	 
	$\size =f_0+f_3-10$
	of a $4$-polytope or $3$-sphere,
	while the $y$-axis represents 
	$f_1+f_2-20-2\cdot\size$,
	so the slope of a line through the
	origin is “fatness$\,-\,2$.”
	
	Black dots~~\raisebox{-.5mm}{
\begin{tikzpicture}[>=stealth']
    \node[name=schwarz, circle, inner sep=1.2mm, draw, fill=black] {};
\end{tikzpicture} }~~mark data points for which Höppner \cite{Hoeppner-dipl}
	had found polytopes.\\
	Grey crossed dots~~\raisebox{-.5mm}{
\begin{tikzpicture}[>=stealth']
	\node[name=kreiss, circle, inner sep=1.2mm, draw, fill=black!25] {};
    \node[name=kreuz, circle, inner sep=1.2mm, draw, cross out] {};
\end{tikzpicture} }~mark 
	coordinates for which $2$-simple $2$-simplicial polytopes
	were found by Paffenholz \& Werner \cite{PaffenholzWerner:many}, and Werner \cite{WernerThesis}.\\
	Grey dots~~\raisebox{-.5mm}{
\begin{tikzpicture}[>=stealth']
    \node[name=gruen, circle, inner sep=1.2mm, draw=black, fill=black!25] {};
\end{tikzpicture} }~give additional
	data points where we now
	found polytopes.\\
	Red dots~~\raisebox{-.5mm}{
\begin{tikzpicture}[>=stealth']
    \node[name=rot, circle, inner sep=1.2mm, draw=black, fill=red] {};
\end{tikzpicture} }~represent coordinates
	of points for which there are $f$-vectors of $3$-spheres,
	but where we found no $f$-vectors of $4$-polytopes;
	left of the dotted line this means that these do not exist.
	
	The graph shown here is complete up to size $12$,
	that is, to the left of the dotted vertical
	line.
	
	White dots~~\raisebox{-.5mm}{
\begin{tikzpicture}[>=stealth']
    \node[name=kreiss,circle, inner sep=1.2mm,  draw, fill=none] {};
\end{tikzpicture} }~appear only to the right of the
	dotted line: They mark locations where the existence of 
	spheres or of polytopes has not been decided. 

\caption{The size/fatness projection of the $f$-vector sets
	 $\fset(\mathcal{P}^4)\subset\fset(\mathcal{S}^3)$}
\end{figure}

Let us note two more intriguing aspects of our enumeration results, which can 
also be seen in Figure~\ref{fig:1}:
\newpage

\begin{observations}\label{obs:ex}~
	\begin{compactenum}[(i)]
	\item
	The sets of ``small'' $f$-vectors $f=(f_0,f_1,f_2,f_3)$ 
	of $3$-spheres and of $4$-polytopes
	differ in an essential way, which is detected by fatness:
	 \begin{compactitem}%
		\item For $\size(f)\le10$, the $f$-vectors of $3$-spheres and of $4$-polytopes agree.
		\item For $\size(f)\le11$, the maximal fatness for $3$-spheres is $4\frac{1}{11}$, 
			                                           for $4$-polytopes it is $4$.
		\item For $\size(f)\le12$, the maximal fatness for $3$-spheres is $4\frac{1}{6}$, 
					                                   for $4$-polytopes it is still~$4$. 
	 \end{compactitem}
	\item In the range of ``small'' $f$-vectors of $\size(f)\le12$,
	the particularly ``fat'' $4$-polytopes are \emph{$2$-simple and $2$-simplicial} 
	in the sense of Grünbaum \cite[Sect.~4.5]{Gruenbaum}. 
	The exceptionally fat $3$-spheres are not $2$-simple and $2$-simplicial, but they still have $f$-vectors
	that are approximately symmetric, with $|f_0-f_3|$ small.
	\end{compactenum}
\end{observations}
 
For this we recall from Grünbaum
\cite[Sect.~4.5]{Gruenbaum} that a $4$-polytope $P$ is $2$-simple and $2$-simplicial (``2s2s'')
if all $2$-faces are triangles both for $P$ and for its dual. 
The definition extends to $3$-spheres and even to Eulerian lattices of length $5$.
Any such 2s2s object has a symmetric $f$-vector, with $f_0=f_3$ and $f_1=f_2$.
The 2s2s property
is detected by the flag vector, but not by the $f$-vector alone.
Observation \ref{obs:ex}(ii) refers to all the polytopes of fatness at least $4$, which
in the range $\size(f)\le12$, according to the classification 
of 2s2s $4$-polytopes and $3$-spheres of size at most $14$ in
Brinkmann \& Ziegler \cite[Thm.~2.1]{BZ_flagvectors}, are
\begin{compactitem}
	\item Werner's example $W_9$ with $9$ vertices \cite[Thm.~4.2.2]{WernerThesis},
	\item $W_{10}$ as well as the hypersimplex $\Delta_4(2)$ and its dual with $10$ vertices, and 
	\item $P_{11}$ by Paffenholz \& Werner \cite[Sect.~4.1]{PaffenholzWerner:many}.
\end{compactitem}
The pattern continues beyond the range $\size(f)\le12$ of our enumeration,
where we find 
\begin{compactitem}
	\item the 2s2s polytope $W_{12}^{39}$ of Werner and Miyata
			\cite[Tbl.~7.1 right]{WernerThesis} \cite[Sect.~4.2]{Miyata-diss} and
	\item the 2s2s sphere $W_{12}^{40}$ with $f$-vector $(12,40,40,12)$
			constructed by Werner \cite[Tbl.~7.1 left]{WernerThesis}.
\end{compactitem}
For this last example we had shown in \cite{BZ_flagvectors}
that it is non-polytopal and that it is the only 2s2s $3$-sphere with such a flag vector.
As a consequence, we
established that the sets of flag vectors of $4$-polytopes and $3$-spheres differ \cite[Theorem 1.1]{BZ_flagvectors}, but did not achieve a similar statement for sets of $f$-vectors.
This is provided by Theorem~\ref{thm:f_differ}.
However, with our new algorithm presented here (plus massive computation)
we also achieved a complete classification result for the $f$-vector $(12,40,40,12)$.

\begin{theorem}\label{thm:12_vert}
	There are $4$ strongly regular $3$-spheres (all of them self-dual, one of them 
	$2$-simple $2$-simplicial),
	but no $4$-polytopes at all, with the $f$-vector $(12,40,40,12)$.
\end{theorem}

So altogether this paper provides six examples of $f$-vectors of $3$-spheres
that are not $f$-vectors of $4$-polytopes, namely the five smallest ones
listed in Theorem~\ref{thm:f_differ} and one more in Theorem~\ref{thm:12_vert}.
Of course one would now want to provide infinitely many examples, 
to show that the cones   $\cone(\fset(\mathcal{P}^4)) \subseteq
  \cone(\fset(\mathcal{S}^3))$ do not coincide, and similar results
for flag vectors and for their cones in higher dimensions.
Our present methods do not seem to provide this.

\section{Objects: Polytopes, Spheres, and Eulerian Lattices}\label{sec:objects} 
 
There have been numerous substantial attempts to 
classify \emph{all} $4$-dimensional polytopes with
some given parameters (e.g., $f$-vectors), or to classify
the parameters that actually occur. 
They all depend on a hierarchy of 
combinatorial/topological/geometric models for convex polytopes
of decreasing generality, which we use systematically in our algorithmic approach.
For basics on convex polytopes, including diagrams/Schlegel diagrams,
we refer to Grünbaum \cite{Gruenbaum} and Ziegler \cite{Z35}.
For regular cell complexes, see Cooke \& Finney \cite{CookeFinney} or Munkres \cite{Munkres:AT}.
Eulerian posets/lattices as combinatorial models arose from the work of Klee \cite{Klee3},
see Stanley \cite[Chap.~3]{Stanley-ec1-2}.
The less common objects we work with can be summarized as follows.

\begin{definition}~
	\begin{compactitem}
		\item A finite graded lattice is \emph{Eulerian} if
		any non-trivial interval has the same number of elements of odd and of even rank;
		it is \emph{interval-connected} if the proper part of any interval of length at least~$3$ is connected.
		\item A CW-sphere is \emph{regular} if the attaching maps of the cells are homeomorphisms also
		on the boundary. The sphere has the \emph{intersection property} if 
		the intersection of any two cells is a single cell (which may be empty). 
	\end{compactitem}
\end{definition}

In this paper we concentrate entirely on the case of $4$-dimensional polytopes,
and correspondingly $3$-spheres and Eulerian lattices of length $5$.
The interval connectivity for Eulerian lattices and the regularity and intersection property for
cellular spheres are always assumed.
We write
\begin{compactitem}
	\item $\mathcal{P}^4$ for the set of combinatorial types of $4$-polytopes; 
	\item $\mathcal{S}^3$ for the set of combinatorial types of $3$-spheres;
	\item $\mathcal{E}^5$ for the isomorphism types of length-$5$ Eulerian lattices.
\end{compactitem}

The boundary complex of any $4$-polytope is a $3$-sphere 
		(regular, cellular, with the intersection property);
the face lattice of any such $3$-sphere is an interval-connected Eulerian lattice of length~$5$.

Polytope theory has produced lots of examples to show that there are strict inclusions
	\[
	\mathcal{P}^4\subset\mathcal{S}^3\subset\mathcal{E}^5
	\]
while there is no difference ``one dimension lower,'' by Steinitz's theorems. 
His theorems also yield that interval-connected Eulerian lattices form an
\emph{excellent} entirely combinatorial model for the topological/geometric structures we are studying.

\begin{proposition}\label{prop:Eulerian5}
	Every interval-connected Eulerian lattice of length $d+1\le 4$ is the face lattice
	of a $d$-polytope. In particular, $\mathcal{P}^3=\mathcal{S}^2=\mathcal{E}^4$.
	
	Every interval-connected Eulerian lattice of length $d+1=5$ is the face lattice
	of a (connected, closed) regular CW $3$-manifold with the intersection property.
\end{proposition}

\begin{proof}[Proof sketch]
	For $d+1\le3$ there is little to prove.
	
	For $d+1=4$ the Eulerian lattice is the face lattice of a connected $2$-manifold
	of Euler characteristic $2$, so we have a sphere.
	The lattice property corresponds to what Steinitz calls
	``Bedingung des Nichtübergreifens''  \cite[S.~179]{SteinitzVorlesung}, 
	which is exactly the intersection
	property for a cellular $2$-sphere.
	Steinitz's Theorem \cite{SteinitzThm,SteinitzVorlesung} yields that
	every such $2$-sphere can be realized as a convex polytope.
	
	For $d+1=5$ the Eulerian lattice is the face poset of a closed connected
	$3$-manifold, whose cells and vertex links are polytopal by Steinitz's theorem.
	However, the fact that this manifold has Euler characteristic $0$
	yields no additional information about its type, by Poincaré duality.
\end{proof}

\section{Enumeration Algorithm}\label{sec:algorithm}

We here propose a new algorithm, which constructs, for 
a given vector $(f_0,f_1,f_2,f_3)$, first the graphs and then
the face lattices of all $3$-manifolds with this $f$-vector, 
by using 0/1 integer programming in order to enumerate all families
of facets that fit to this graph and all other constraints. The algorithm has the following outline:
\begin{alg}\label{algorithm}
 \verb+find_lattices(f)+\\
 \textup{INPUT:} A vector $(f_0,f_1,f_2,f_3)\in\Z^4$\\
 \textup{OUTPUT:} All Eulerian lattices of length $5$ with this $f$-vector
\begin{compactenum}[\rm(i)]
	\item enumerate all graphs $G$ on $f_0$ vertices and $f_1$ edges that are $4$-connected;
	\item for every graph $G$ find all induced subgraphs that are planar and $3$-connected;
	\item construct for every graph $G$ an integer program (IP) with binary variables 
		corresponding to the possible
		facets and ridges (faces of the facets), and with constraints given by the $f$-vector, proper
		intersection, the Euler relation, and the graph;
	\item enumerate all feasible solutions of this IP;
	\item check for every feasible solution whether it gives an Eulerian lattice.
\end{compactenum}
\end{alg}

\begin{proposition}
	Algorithm~\ref{algorithm} enumerates all interval-connected Eulerian lattices of length $5$ with $f$-vector $(f_0,f_1,f_2,f_3)$.
\end{proposition}

\begin{proof}
	We rely on the interpretation of interval-connected length $5$ Eulerian lattices 
	as face lattices of cellular regular $3$-manifolds with intersection property in Proposition~\ref{prop:Eulerian5}.
	Since the graph of any such manifold is $4$-connected, 
	Step~(i) will not exclude any graph of some $3$-manifold
	with $f$-vector $(f_0,f_1,f_2,f_3)$.
	
	Also by Proposition~\ref{prop:Eulerian5} the graphs of interval-connected Eulerian lattices of
	length $4$ (and thus of facets of cellular $4$-manifolds) are exactly the planar and $3$-connected
	graphs. Thus, with Step~(ii) we find a list $\mathcal{F}_G$ of all potential facets for a manifold with the given graph $G$.
	
	From the list $\mathcal{F}_G$, we also get the list $\mathcal{R}_G$
	of the potential ridges, simply from the faces of the facets. We now construct a 0/1-IP 
	whose variables $x_i$ represent the facets $F_i$, and the variables $y_j$ the ridges $R_j$, 
	such that all solutions correspond to pseudomanifolds formed by a subset of the facets in $\mathcal{F}_G$ and 
	such that all face lattices of $3$-manifolds with graph $G$ and $f$-vector $(f_0,f_1,f_2,f_3)$
	are feasible solutions, with the constraints
	\begin{eqnarray}
		\sum_i x_i &=& f_3\label{eq:IP_f3}\\
		\sum_j y_j &=& f_2\label{eq:IP_f2}\\
		2y_j-\sum_{F_i\supset R_j} x_i &=& 0\qquad\textrm{for all ridges }R_j\label{eq:IP_ridges}\\
		x_i,y_j&\in &\{0,1\}.\label{eq:IP_binary}
	\end{eqnarray}
	Condition~(\ref{eq:IP_binary}) says that
	all variables are binary, which means that if a variable in the solution is $1$ the corresponding face
	will selected.
	Equations~(\ref{eq:IP_f3}) and (\ref{eq:IP_f2}) enforce that the total number of facets and ridges selected is
	$f_3$, resp.~$f_2$. Equation (\ref{eq:IP_ridges}) ensures that ridge $R_j$ is used if and only if
	precisely two facets containing it are selected.  Similarly, we get constraints from the Euler relation for the
	intervals above the vertices and edges, such that all feasible solutions correspond to 
	Eulerian posets. Moreover, for every edge we get an inequality forcing the number
	of faces containing it to be larger than zero. Finally, we get inequalities $x_i+x_j\leq 1$ for pairs
	of facets $F_i,F_j$ if their intersection is not \emph{proper} (i.e.~that not both can appear in a
	$3$-manifold simultaneously).
	Since the face lattice of any $3$-manifold with the given $f$-vector and graph $G$ satisfies 
	the constraints of the IP, it will be in the set of feasible solutions of this IP. Therefore, with
	the last step we can complete the enumeration of all interval-connected Eulerian lattices
	with the given $f$-vector.
\end{proof}

We implemented Algorithm \ref{algorithm}
in \emph{sage} \cite{sage}, using the \emph{geng}-function of \emph{nauty}~\cite{nauty} 
(which is a built-in function of \emph{sage}) to enumerate all graphs on $f_0$ vertices, 
with $f_1$ edges, with minimal vertex degree at least $4$, and being $2$-connected 
(\emph{nauty} cannot enumerate $4$-connected graphs, so we had to relax to $2$-connectedness, 
but this did not include too many extra graphs), and the MILP-library of \emph{sage} 
to check the IPs for feasibility and to enumerate all their solutions. 
We enumerated all feasible solutions iteratively: Given a feasible solution, we store it 
and set the sum of the $f_3$ variables corresponding to the facets of this solution to be 
at most $f_3-1$. Thus, we excluded with an additional constraint precisely the solution we 
just found and optimized again. By iterating this until no feasible solution remained, we 
enumerated all feasible solutions of the original IP.

Finally we had to check every solution to represent an interval-connected Eulerian lattice of length~$5$:
By construction, we were looking at Eulerian posets.
For each of these interval-connectivity was easy to check, as was the intersection property:
Both these properties were not completely built into our IP.
Then we triangulated the corresponding manifold and used \emph{sage} to calculate the Betti numbers,
and thus verified that in all cases considered we were dealing with homology spheres.
Then we used \emph{BI\-STELLAR} by Lutz~\cite{bistellar} to show that each of them was 
flip-equivalent to the boundary of the simplex, and thus a genuine sphere. 

\section{Enumeration and Classification Results}\label{sec:classification}

For the proof of Theorem~\ref{thm:f_differ},
we started with the generation of all potential flag-vectors
bounded by  $f_0+f_3\le22$, 
that is, all integer vectors $(f_0,f_1,f_2,f_3;f_{03})\in\Z^5$
that satisfy all the linear and non-linear conditions on $f$-vectors 
and of flag vectors
that were known to be valid for Eulerian lattices with the intersection property
of length~$5$, as given by
Barnette \cite{barnette72:_inequal},
Bayer \cite{Bay},
and Ling \cite{ling:flag-vectors}.
(See Bayer \& Lee \cite{BaLee} and Höppner \& Ziegler \cite{HoeppnerZie} for surveys.)
Moreover, as we in the algorithmic approach started to add to the $f$-vector information
specific data about the combinatorial types of facets used, we could make use of
constraints such as 
\begin{eqnarray*}
f_{02} - 4f_2+3f_1-2f_0&\leq& \tbinom{f_0}{2}- \tfrac{1}{2} \sum_{{F\,\textrm{facet}},\ {f_0(F)\geq 7}} (m_i(F)+f_{02}(F)-3f_2(F))\\
&& - \#\textup{facets with 6 vertices}+ \tfrac{1}{2}\#\textup{pyramids over pentagon},
\end{eqnarray*}
where $m_i(F)$ denotes the number of interior edges of a face $F$,
proved in Brinkmann \cite[Sect.~2.2.1]{B_thesis},
which sharpens an inequality by Bayer \cite{Bay}.

	Moreover, we could (and did) assume that $f_0,f_3\ge9$,
	as the objects with up to $8$ vertices and facets have been enumerated
	and analyzed in detail by Altshuler \& Steinberg~\cite{Spheres8v1,Spheres8v2}.
	
	Furthermore, we ticked off on our candidate list all those vectors that are known
	to occur as $f$-vectors of $4$-polytopes, 
	for example from the study of Höppner \& Ziegler \cite{HoeppnerZie}  
	or the enumeration of 2s2s-polytopes in Brinkmann \& Ziegler \cite{BZ_flagvectors}.
	
	For all remaining candidate vectors we  
	enumerated all compatible Eulerian lattices by Algorithm \verb+find_lattices(f)+,
	and then used the methods detailed in Brinkmann \& Ziegler \cite{BZ_flagvectors}
in order to  
\begin{compactitem}
	\item either use first numerical non-linear optimization techniques and then
		exact arithmetic sharpenings in order to find rational coordinates for at least one polytope 
		with the given $f$-vector,
	\item or use biquadratic final polynomials for partial oriented matroids
	 	 in order to prove that \emph{all} spheres for the given $f$-vector are non-realizable.
\end{compactitem}
The results are shown in Table \ref{tbl:enumeration_results_12}:
It lists, for each potential $f$-vector considered, 
the number of graphs to be checked 
(graphs on $f_0$ vertices, with $f_1$ edges, with minimal vertex degree at least~$4$, and being $2$-connected),
and the numbers 
\begin{compactitem}
	\item $\#\mathcal{E}^5$ of Eulerian lattices of length $5$,
	\item $\#\mathcal{S}^3$ of cellular $3$-spheres, 
	\item $\#$np  of non-polytopal $3$-spheres among them, and
	\item $\#\mathcal{P}^4$ of convex $4$-polytopes
\end{compactitem}
with the given $f$-vector.
In some instances for the last two quantitites we just give lower bounds,
if we did not decide all cases.
An asterisk $*$ marks objects where we have exact coordinates for at least one polytope and 
approximate (floating point) coordinates for the others. 
Blank spaces represent missing data (e.g.~not enumerated/\allowbreak calculated). 
In particular for $f_0=11$ we did not enumerate all $f$-vectors, but restricted ourselves to constructing polytopes.

\begin{longtable}{|c|c|c|c|c|c|X|}
\multicolumn{7}{c}
{{\bfseries Table \ref{tbl:enumeration_results}}} \\
\hline
$f$-vector & $\#$ graphs & $\#\mathcal{E}^5$ & $\#\mathcal{S}^3$ & $\#$np & $\#\mathcal{P}^4$ & \\ 
\hline
\endfirsthead
\multicolumn{7}{c}%
{{\bfseries Table \ref{tbl:enumeration_results} -- continued from previous page}} \\
\hline
$f$-vector & $\#$ graphs & $\#\mathcal{E}^5$ & $\#\mathcal{S}^3$ & $\#$np & $\#\mathcal{P}^4$ & \\ 
\hline
\endhead
\hline
\endfoot
\hline
\caption{All potential $f$-vectors with $f_0,f_3\ge9$ and $f_0+f_3\le22$.}
\label{tbl:enumeration_results}
\endlastfoot

$(9,m,m,9)$ & $170$ & $0$ & $0$ & $0$ & $0$ & $m\leq 19$\\
$(9,20,20,9)$ & $713$ & $1$ & $1$ & $0$ & $1$ &\\
$(9,21,21,9)$ & $1\,754$ & $0$ & $0$ & $0$ & $0$ &\\
$(9,22,22,9)$ & $2\,770$ & $129$ & $129$ & & $\geq 54$ &\\
$(9, 23, 23, 9)$ & $3\,129$ & $211$ & $211$ & $\geq 2$ & $\geq 113^*$ &\\
$(9, 24, 24, 9)$ & $2\,723$ & $118$ & $118$ & $\geq 2$ & $\geq 81^*$ &\\
$(9, 25, 25, 9)$ & $1\,917$ & $7$ & $7$ & $0$ & $7^*$ &\\
$(9,26,26,9)$ & $1\,154$ & $1$ & $1$ & $0$ & $1$ & $W_9$ \cite[Thm.~4.2.2]{WernerThesis}\\
$(9,m,m,9)$ & $1\,132$ & $0$ & $0$ & $0$ & $0$ & $m\geq 27$\\
$(9,m,m+1,10)$ & $2\,673$ & $0$ & $0$ & $0$ & $0$ & $m\leq 21$\\
$(9, 22, 23, 10)$ & $2\,770$ & $12$ & $12$ &  & $\geq 9^*$ &\\
$(9, 23, 24, 10)$ & $3\,129$ & $398$ & $398$ & $\geq 1$ & $\geq 78^*$ & \\
$(9, 24, 25, 10)$ & $2\,723$ & $904$ & $904$ & $\geq 7$ & $\geq 27^*$ & \\
$(9, 25, 26, 10)$ & $1\,917$ & $524$ & $524$ & $\geq 15$ & $\geq 80^*$ &\\
$(9, 26, 27, 10)$ & $1\,154$ & $67$ & $67$ & $\geq 2$ & $\geq 62^*$ & \\
$(9, 27, 28, 10)$ & $610$ & $0$ & $0$ & $0$ & $0$ & \\
$(9, 28, 29, 10)$ & $294$ & $0$ & $0$ & $0$ & $0$ & \\
$(9, 29, 30, 10)$ & $133$ & $0$ & $0$ & $0$ & $0$ & \\
$(9,m,m+1,10)$ & $95$ & $0$ & $0$ & $0$ & $0$ & $m\geq 30$\\
$(9, m, m+2, 11)$ & $5\,443$ & $0$ & $0$ & $0$ & $0$ & $m\leq 22$\\
$(9, 23, 25, 11)$ & $3\,129$ & $66$ & $66$ & & $\geq 34$ & \\
$(9, 24, 26, 11)$ & $2\,723$ & $1\,188$ & $1\,188$ & & $\geq 105$ & \\
$(9, 25, 27, 11)$ & $1\,917$ & $2\,650$ & $2\,650$ & & $\geq 52$ & \\
$(9, 26, 28, 11)$ & $1\,154$ & $1\,344$ & $1\,344$ & & $\geq 1$ & \\
$(9, 27, 29, 11)$ & $610$ & $125$ & $125$ & & $\geq 60$ & \\
$(9, 28, 30, 11)$ & $294$ & $3$ & $3$ & $1$ & $2$ & \\
$(9, 29, 31, 11)$ & $133$ & $0$ & $0$ & $0$ & $0$ & \\
$(9, m, m+2, 11)$ & $103$ & $0$ & $0$ & $0$ & $0$ & $m\geq 30$\\
$(9, m, m+3, 12)$ & $5\,443$ & $0$ & $0$ & $0$ & $0$ & $m\leq 22$\\
$(9, 23, 26, 12)$ & $3\,129$ & $3$ & $3$ & $0$ & $3$ & \\
$(9, 24, 27, 12)$ & $2\,723$ & $335$ & $335$ & & $\geq 129$ & \\
$(9, 25, 28, 12)$ & $1\,917$ & $3\,275$ & $3\,275$ & & $\geq 171$ &\\
$(9, 26, 29, 12)$ & $1\,154$ & $5\,928$ & $5\,928$ & & $\geq 276$ &\\
$(9, 27, 30, 12)$ & $610$ & $2\,171$ & $2\,171$ & & $\geq 516$ & \\
$(9, 28, 31, 12)$ & $294$ & $113$ & $113$ & & $\geq 33$ & \\
$(9, 29, 32, 12)$ & $133$ & $0$ & $0$ & $0$ & $0$ & \\
$(9, 30, 33, 12)$ & $59$ & $0$ & $0$ & $0$ & $0$ & \\
$(9, m, m+3, 12)$ & $44$ & $0$ & $0$ & $0$ & $0$ & $m\geq 31$\\
$(9,m,m+4,13)$ & $8\,536$ & $0$ & $0$ & $0$ & $0$ & $m\leq 23$\\
$(9, 24, 28, 13)$ & $2\,723$ & $33$ & $33$ & & $\geq 32^*$ & \\
$(9, 25, 29, 13)$ & $1\,917$  & $1\,223$ & $1\,223$ & $\geq 1$ & $\geq 387^*$ &\\
$(9, 26, 30, 13)$ & $1\,154$ & $7\,677$ & $7\,677$ & $\geq 3$ & $\geq 309$ & \\
$(9, 27, 31, 13)$ & $610$ & $9\,773$ & $9\,773$ & $\geq 32$ & $\geq 13$ & \\
$(9, 28, 32, 13)$ & $294$ & $2\,136$ & $2\,136$ &  & $\geq 439$ & \\
$(9, 29, 33, 13)$ & $133$ & $27$ & $27$ & $\geq 1$ & $\geq 9^*$ & \\
$(9, 30, 34, 13)$ & $59$ & $0$ & $0$ & $0$ & $0$ & \\
$(9,m,m+4,13)$ & $44$ & $0$ & $0$ & $0$ & $0$ & $m\geq 31$\\
\hline
$(10,m,m,10)$ & $10\,247$ & $0$ & $0$ & $0$ & $0$ & $m\leq 22$\\
$(10, 23, 23, 10)$ & $35\,219$ & $4$ & $4$ & $0$ & $4$ & \\
$(10, 24, 24, 10)$ & $87\,014$ & $16$ & $16$ &  & $\geq 2^*$ & \\
$(10,25,25,10)$ & $152\,369$ &&&& $\geq 296$ & pyramids\\
$(10, 26, 26, 10)$ & $203\,469$ & $5\,550$ & $5\,550$ & $\geq 69$ & $\geq 2^*$ &\\
$(10, 27, 27, 10)$ & $217\,596$ & $5\,561$ & $5\,561$ & $\geq 204$ & $\geq 90^*$ &\\
$(10, 28, 28, 10)$ & $192\,964$ & $1\,662$ & $1\,662$ & $\geq 143$ & $\geq 13^*$ &\\
$(10, 29, 29, 10)$ & $145\,773$ & $128$ & $128$ & $\geq 2$ & $\geq 21^*$ & \\
$(10, 30, 30, 10)$ & $95\,827$ & $3$ & $3$ & $0$ & $3$ & $\Delta_4(2)$, $\Delta_4(2)^*$, $W_{10}$\\
$(10, 31, 31, 10)$ & $55\,762$ & $0$ & $0$ & $0$ & $0$ &\\
$(10,m,m,10)$ & $53\,718$ & $0$ & $0$ & $0$ & $0$ & $m\geq 32$\\
$(10,m,m+1,11)$ & $45\,469$ & $0$ & $0$ & $0$ & $0$ & $m\leq 23$\\
$(10, 24, 25, 11)$ & $87\,014$ & $6$ & $6$ &  & $\geq 2^*$ & \\
$(10, 25, 26, 11)$ & $152\,369$ & $136$ & $136$ & & $\geq 10^*$ & \\
$(10, 26, 27, 11)$ & $203\,469$ & $6\,794$ & $6\,794$ & $\geq 11$ & $\geq 633$ & \\
$(10, 27, 28, 11)$ & $217\,596$ & $24\,915$ & $24\,915$ &  & $\geq 22$& \\
$(10, 28, 29, 11)$ & $192\,964$ & $30\,355$ & $30\,355$ & $\geq 1$ & $\geq 159$ & \\
$(10, 29, 30, 11)$ & $145\,773$ & $11\,916$ & $11\,916$ &  & $\geq 28$ & \\
$(10, 30, 31, 11)$ & $95\,827$ & $1\,441$ & $1\,441$ & $\geq 61$ & $\geq 1$ & \\
$(10, 31, 32, 11)$ & $55\,762$ & $35$ & $35$ & $\geq 9$ & $\geq 20^*$ & \\
$(10, 32, 33, 11)$ & $29\,199$ & $2$ & \textbf{2} & \textbf{2} & \textbf{0} &\\
$(10, 33, 34, 11)$ & $13\,981$ & $0$ & $0$ & $0$ & $0$ & \\
$(10, 34, 35, 11)$ & $6\,202$ & $0$ & $0$ & $0$ & $0$ & \\
$(10, 35, 36, 11)$ & $2\,600$ & $0$ & $0$ & $0$ & $0$ & \\
$(10,m,m+1,11)$ & $1\,736$ & $0$ & $0$ & $0$ & $0$ & $m\geq 36$\\
$(10,m,m+2,12)$ & $45\,469$ & $0$ & $0$ & $0$ & $0$ & $m\leq 23$\\
$(10, 24, 26, 12)$ & $87\,014$ & $2$ & $2$ &  & $\geq 1$ & \\
$(10, 25, 27, 12)$ & $152\,369$ & $2$ & $2$ &  & $\geq 1$ & \\
$(10, 26, 28, 12)$ & $203\,469$ & $1\,051$ & $1\,051$ &  & $\geq 178$ & \\
$(10, 27, 29, 12)$ & $217\,596$ & $23\,884$ & $23\,884$ &  & $\geq 768$ &\\
$(10, 28, 30, 12)$ & $192\,964$ & $91\,727$ & $91\,727$ &  & $\geq 455$ &\\
$(10, 29, 31, 12)$ & $145\,773$ & $112\,266$ & $112\,266$ &  & $\geq 256$ & \\
$(10, 30, 32, 12)$ & $95\,827$ & $47\,141$ & $47\,141$ & $\geq 13$ & $\geq 1$ & \\
$(10, 31, 33, 12)$ & $55\,762$ & $5\,943$ & $5\,943$ & $\geq 521$ & $\geq 368$ &\\
$(10, 32, 34, 12)$ & $29\,199$ & $225$ & $225$ &  & $\geq 7$ & \\
$(10, 33, 35, 12)$ & $13\,981$ & $1$ & \textbf{1} & \textbf{1} & \textbf{0} &\\
$(10, 34, 36, 12)$ & $6\,202$ & $0$ & $0$ & $0$ & $0$ & \\
$(10, 35, 37, 12)$ & $2\,600$ & $0$ & $0$ & $0$ & $0$ & \\
$(10,m,m+2,12)$ & $1\,736$ & $0$ & $0$ & $0$ & $0$ & $m\geq 36$\\
\hline
$(11,22,22,11)$ & $265$ & $0$ & $0$ & $0$ & $0$ &\\
$(11,23,23,11)$ & $10\,391$ & $0$ & $0$ & $0$ & $0$ &\\
$(11,24,24,11)$ & $120\,985$ & $0$ & $0$ & $0$ & $0$ &\\
$(11,25,25,11)$ & $696\,184$ & $0$ & $0$ & $0$ & $0$ &\\
$(11,26,26,11)$ & $2\,504\,998$ & $21$ & $21$ & & $\geq 1$ & \\
$(11,27,27,11)$ & $6\,383\,318$ & $322$ & $322$ &  & $\geq 1$ & \\
$(11,28,28,11)$ & $12\,417\,723$ &&&& $\geq 2635$ & pyramids\\
$(11,29,29,11)$ & $19\,379\,000$ &  &  &  & $\geq 1$ & \\
$(11,30,30,11)$ & $25\,121\,426$ &  &  &  & $\geq 1$ & \\
$(11,31,31,11)$ & $27\,749\,332$ &  &  &  & $\geq 1$ & \\
$(11,32,32,11)$ & $26\,626\,961$ &  &  &  & $\geq 104$ & \\
$(11,33,33,11)$ & $22\,528\,512$ &  &  &  & $\geq 1$ & \\
$(11,34,34,11)$ & $17\,005\,570$ & $100$ & $100$ & $\geq 15$ & $\geq 1$ & $P_{11}$\\
$(11,35,35,11)$ & $11\,561\,155$ & $2$ & \textbf{2} & \textbf{2} & \textbf{0} & \\
$(11,36,36,11)$ & $7\,134\,337$ & $0$ & $0$ & $0$ & $0$ &\\
\end{longtable}

The spheres with the particular $f$-vectors 
$(10,32,33,11)$, 
$(10,33,35,12)$, and 
$(11,35,35,11)$ 
of Theorem~\ref{thm:f_differ}
will be presented and discussed in Section~\ref{sec:examples}.
\bigskip

The proof of Theorem~\ref{thm:12_vert} follows the same pattern,
with considerably higher computation times.
Table \ref{tbl:enumeration_results_12} shows the results of the computation for the 
potential $f$-vectors $(12,m,m,12)$ for large $m$: The numbers of graphs to check (graphs on $f_0$ vertices, with $f_1$ edges, with minimal vertex degree at least $4$, $2$-connected) and the numbers of strongly regular $3$-manifolds, strongly regular $3$-spheres, non-polytopal spheres, and $4$-polytopes.  
Blank spaces represent missing data (e.g.~not enumerated or calculated). For time reasons, and since there is a polytope, we did not enumerate the manifolds with $f$-vector $(12,39,39,12)$. 
  The results for larger $m$ follow as any manifold with such an $f$-vector would be 2s2s, as verified in Brinkmann \cite[Prop.~2.2.19]{B_thesis},
  and these we have enumerated, see Brinkmann \& Ziegler \cite[Thm.~2.1]{BZ_flagvectors}.  

\begin{longtable}{|c|c|c|c|c|c|X|}
\multicolumn{7}{c}
{{\bfseries Table \ref{tbl:enumeration_results_12}}} \\
\hline
$f$-vector & $\#$ graphs & $\#\mathcal{E}^5$ & $\#\mathcal{S}^3$ & $\#$np & $\#\mathcal{P}^4$ & \\ 
\hline
\endfirsthead
\multicolumn{7}{c}%
{{\bfseries Table \ref{tbl:enumeration_results_12} -- continued from previous page}} \\
\hline
$f$-vector & $\#$ graphs & $\#\mathcal{E}^5$ & $\#\mathcal{S}^3$ & $\#$np & $\#\mathcal{P}^4$ & \\ 
\hline
\endhead
\hline
\endfoot
\hline 
\caption{Results for the potential $f$-vectors $(12,40,40,12)$ and $(12,41,41,12)$}
\label{tbl:enumeration_results_12}
\endlastfoot
 
$(12,39,39,12)$ & $4\,078\,410\,035$ & $\geq 1$ & $\geq 1$ & & $\geq 1$ & $W_{12}^{39}$\\
$(12,40,40,12)$ & $2\,997\,683\,218$& $4$ & \textbf{4} & \textbf{4} & \textbf{0} & \\
$(12,41,41,12)$ & $2\,037\,876\,411$ & $0$ & $0$ & $0$ & $0$ & \\
$(12,m,m,12)$ & $4\,880\,253\,668$ & $0$ & $0$ & $0$ & $0$ & $m\geq 42$\\
\end{longtable}

\section{Examples}\label{sec:examples}

According to Theorem \ref{thm:f_differ} there are five $f$-vectors for which there 
is at least one $3$-sphere but no $4$-polytope. 
In this section we will present these $3$-spheres.
For each of these $f$-vectors,
\begin{compactitem}
	\item the fact that there are no other $3$-spheres than those we present in the following
		depends on massive computation and does not seem to have a reasonably short or ``compact'' proof,
	\item the fact that the objects that we present are, indeed, spheres, can be verified in a
		variety of ways; in the following we present coordinates and images for a diagram (in the sense
			of polytope theory, see Ziegler \cite[Lect.~5]{Z35}),
	\item the fact that the spheres are not polytopal was verified on the computer with oriented matroid
		techniques; in principle, one can extract human-verifiable short proofs from the computation results;
		for this we give one example below. 
\end{compactitem}


\begin{examples}\label{ex:10_32_33_11}
	There are two $3$-spheres with $f$-vector $(10,32,33,11)$:
	\begin{compactitem}
		\item 
			The sphere $(10_{32,33}^0)$ is given by the facet list
			\smallskip
			
			\begin{minipage}{0.49\textwidth} 
			$F_{\hz0} = \{v_{\hz0}, v_{\hz2}, v_{\hz4}, v_{\hz5}, v_{\hz9}\}$\\ 
			$F_{\hz1} = \{v_{\hz0}, v_{\hz2}, v_{\hz4}, v_{\hz6}, v_{\hz8}\}$\\ 
			$F_{\hz2} = \{v_{\hz1}, v_{\hz3}, v_{\hz6}, v_{\hz7}, v_{\hz9}\}$\\ 
			$F_{\hz3} = \{v_{\hz1}, v_{\hz3}, v_{\hz4}, v_{\hz6}, v_{\hz8}\}$\\ 
			$F_{\hz4} = \{v_{\hz0}, v_{\hz2}, v_{\hz5}, v_{\hz7}, v_{\hz8}\}$\\
			$F_{\hz5} = \{v_{\hz1}, v_{\hz3}, v_{\hz5}, v_{\hz7}, v_{\hz8}\}$ 
			\end{minipage}
			\begin{minipage}{0.49\textwidth} 
			
			$F_{\hz6} = \{v_{\hz0}, v_{\hz1}, v_{\hz4}, v_{\hz6}, v_{\hz9}\}$\\ 
			$F_{\hz7} = \{v_{\hz2}, v_{\hz3}, v_{\hz5}, v_{\hz7}, v_{\hz9}\}$\\ 
			$F_{\hz8} = \{v_{\hz1}, v_{\hz2}, v_{\hz4}, v_{\hz7}, v_{\hz8}\}$\\ 
			$F_{\hz9} = \{v_{\hz1}, v_{\hz2}, v_{\hz4}, v_{\hz7}, v_{\hz9}\}$\\ 
			$F_{10} = \{v_{\hz0}, v_{\hz3}, v_{\hz5}, v_{\hz6}, v_{\hz8}, v_{\hz9}\}$\\
			\end{minipage}
			\smallskip
			
			It is non-polytopal, but it has diagrams based on each of the facets $F_0$, $F_1$, $F_2$, $F_3$, $F_4$, $F_5$, $F_6$, and $F_7$, 
			but not based on one of $F_8$, $F_9$, or $F_{10}$. 
			A diagram based on facet $F_2$ is given in Figure~\ref{fig:diagram_10_32_33_11_0}.
			This sphere cannot be realized by a fan, and thus it is not  
			star-shaped in the sense of Ewald \cite[Sect.~III.5]{Ewal}.
		\item 
 			The sphere $(10_{32,33}^1)$ is given by the facet list
 			\smallskip
			
			\begin{minipage}{0.49\textwidth} 
			$F_{\hz0} = \{v_{\hz0}, v_{\hz3}, v_{\hz5}, v_{\hz6}, v_{\hz8}\}$\\ 
			$F_{\hz1} = \{v_{\hz0}, v_{\hz4}, v_{\hz5}, v_{\hz7}, v_{\hz8}\}$\\ 
			$F_{\hz2} = \{v_{\hz0}, v_{\hz3}, v_{\hz4}, v_{\hz6}, v_{\hz7}\}$\\ 
			$F_{\hz3} = \{v_{\hz0}, v_{\hz1}, v_{\hz3}, v_{\hz5}, v_{\hz7}\}$\\ 
			$F_{\hz4} = \{v_{\hz1}, v_{\hz3}, v_{\hz5}, v_{\hz8}, v_{\hz9}\}$\\
			$F_{\hz5} = \{v_{\hz1}, v_{\hz3}, v_{\hz6}, v_{\hz7}, v_{\hz9}\}$ 
			\end{minipage}
			\begin{minipage}{0.49\textwidth}
			$F_{\hz6} = \{v_{\hz0}, v_{\hz2}, v_{\hz4}, v_{\hz6}, v_{\hz8}\}$\\ 
			$F_{\hz7} = \{v_{\hz2}, v_{\hz4}, v_{\hz6}, v_{\hz7}, v_{\hz9}\}$\\ 
			$F_{\hz8} = \{v_{\hz2}, v_{\hz4}, v_{\hz5}, v_{\hz8}, v_{\hz9}\}$\\ 
			$F_{\hz9} = \{v_{\hz1}, v_{\hz4}, v_{\hz5}, v_{\hz7}, v_{\hz9}\}$\\ 
			$F_{10} = \{v_{\hz2}, v_{\hz3}, v_{\hz6}, v_{\hz8}, v_{\hz9}\}$\\
			\end{minipage}
			\smallskip
			
			It is non-polytopal, but it has a diagram based on every facet and it can be represented 
			by a fan. 
			A diagram based on facet $F_2$ is given in Figure~\ref{fig:diagram_10_32_33_11_1}.
	\end{compactitem}
\end{examples}			

\begin{figure}[ht]
\centering 
\begin{minipage}{0.3\textwidth}
\vspace*{6pt}
$\underline{F_2}$\\
$v_{\hz0}  =  (906, 197, 915)$ \\ 
$v_{\hz1}  =  (228623/5810, 18, 986)$ \\ 
$v_{\hz2}  =  (90, 942, 119)$ \\ 
$v_{\hz3}  =  (983, 18, 10)$ \\ 
$v_{\hz4}  =  (485, 502, 941)$ \\ 
$v_{\hz5}  =  (448, 647, 296)$ \\ 
$v_{\hz6}  =  (974, 18, 908)$ \\ 
$v_{\hz7}  =  (18, 977, 14)$ \\ 
$v_{\hz8}  =  (665, 333, 592)$ \\ 
$v_{\hz9}  =  (983, 990, 985)$
\end{minipage}
\begin{minipage}{0.49\textwidth}
\includegraphics[scale=0.63]{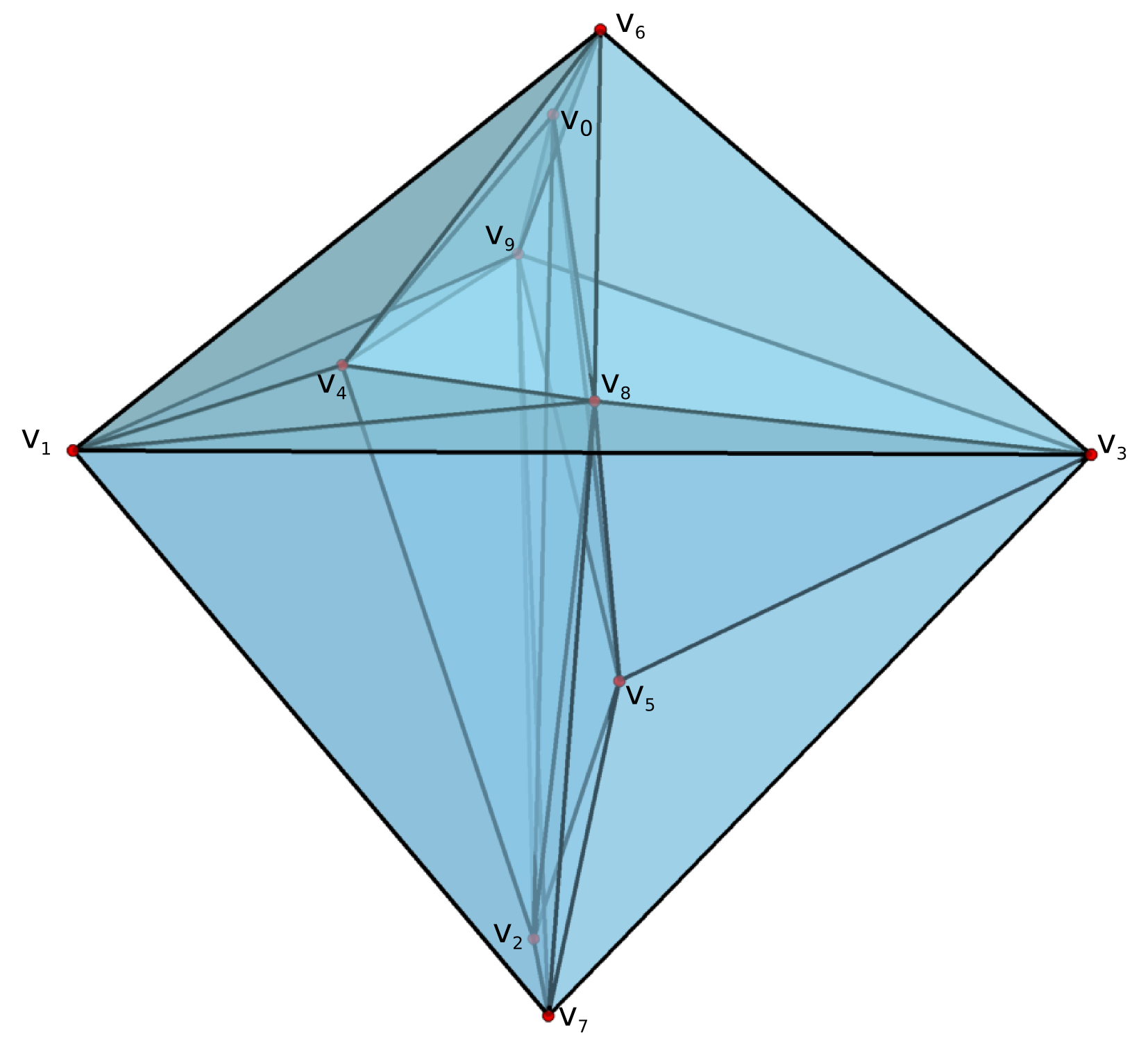}
\end{minipage}
\caption{A diagram based on facet $F_2$ for the sphere $(10_{32,33}^0)$ with $f$-vector $(10,32,33,11)$.}
\label{fig:diagram_10_32_33_11_0}
\end{figure}

\begin{figure}
\centering
\begin{minipage}{0.3\textwidth}
\vspace*{5pt}
$\underline{F_0}$\\
$v_{\hz0}  =  (11, 10, 26)$ \\ 
$v_{\hz1}  =  (13, 16, 10)$ \\ 
$v_{\hz2}  =  (9, 10, 11)$ \\ 
$v_{\hz3}  =  (16, 8, 10)$ \\ 
$v_{\hz4}  =  (9, 11, 14)$ \\ 
$v_{\hz5}  =  (12, 24, 9)$ \\ 
$v_{\hz6}  =  (11, 9, 11)$ \\ 
$v_{\hz7}  =  (11, 14, 16)$ \\ 
$v_{\hz8}  =  (5, 10, 8)$ \\ 
$v_{\hz9}  =  (11, 13, 10)$
\end{minipage}
\begin{minipage}{0.49\textwidth}
\includegraphics[scale=0.6]{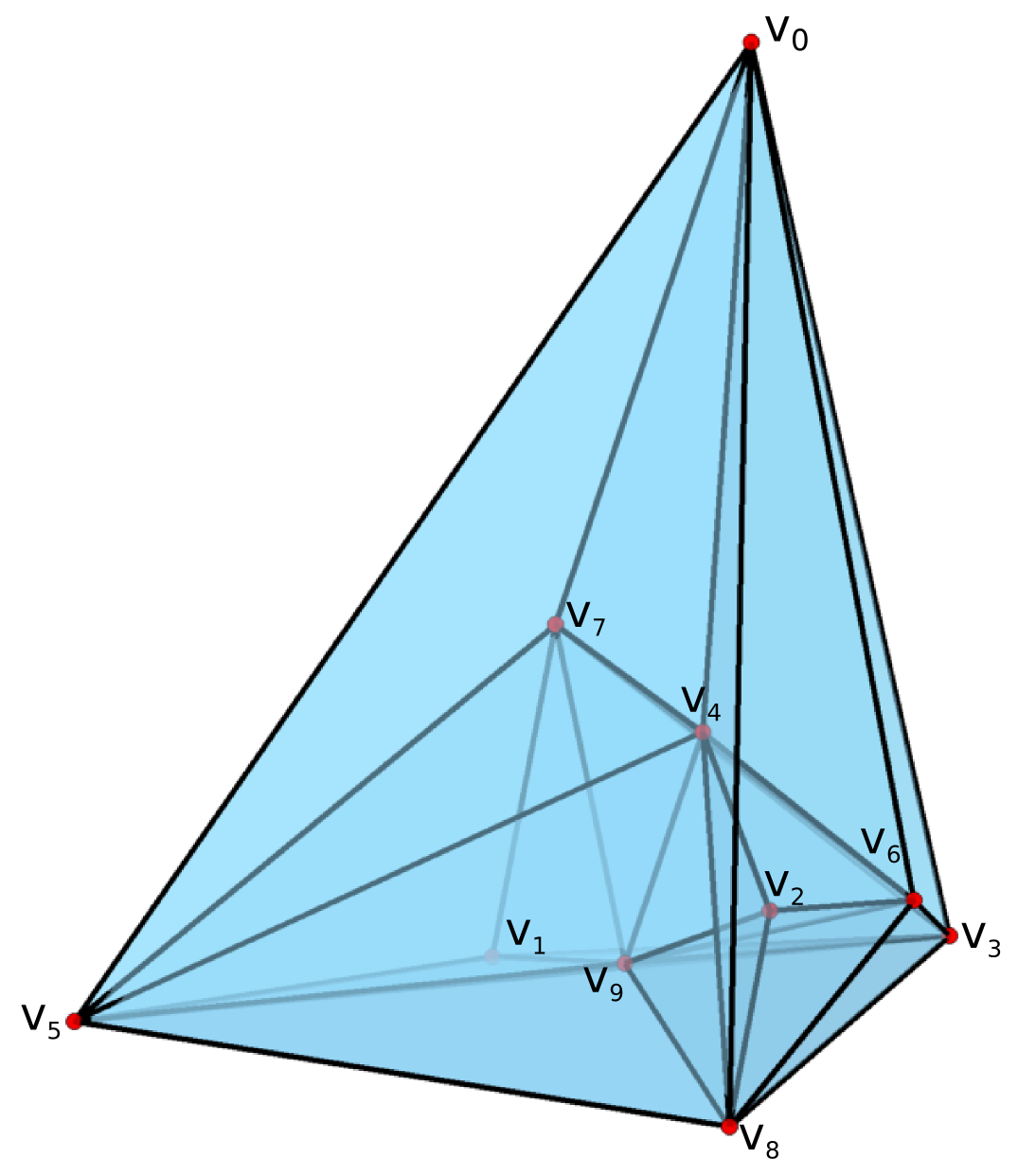}
\end{minipage}
\caption{A diagram based on facet $F_0$ for the sphere $(10_{32,33}^1)$ with $f$-vector $(10,32,33,11)$.}
\label{fig:diagram_10_32_33_11_1}
\end{figure}

We did not manage to decide whether the second sphere $(10_{32,33}^1)$ has a star-shaped embedding.
An oriented matroid that would support such an embedding exists.
(Clearly every star-shaped sphere can be represented by a fan. The converse is true for simplicial spheres,
but not in general.)

The following is an example for a human-verifiable non-polytopality proof, for the first sphere in
Proposition~\ref{ex:10_32_33_11}. 
The non-existence proofs for diagrams use the same technique.
For more details and more examples see Brinkmann \cite{B_thesis}.

\begin{proposition}\label{prop:10_32_33_11_a_np}
The sphere $(10_{32,33}^0)$ is non-polytopal.
\end{proposition}
 
\begin{proof}
We will use a similar oriented matroid approach as in~\cite{BZ_flagvectors}. 
The following arguments may be verified with reference to the 
list of labeled facets displayed in Figure~\ref{fig:10_32_33_11_a_facets}.

\begin{figure}[ht]
\centering
\includegraphics[scale=0.9]{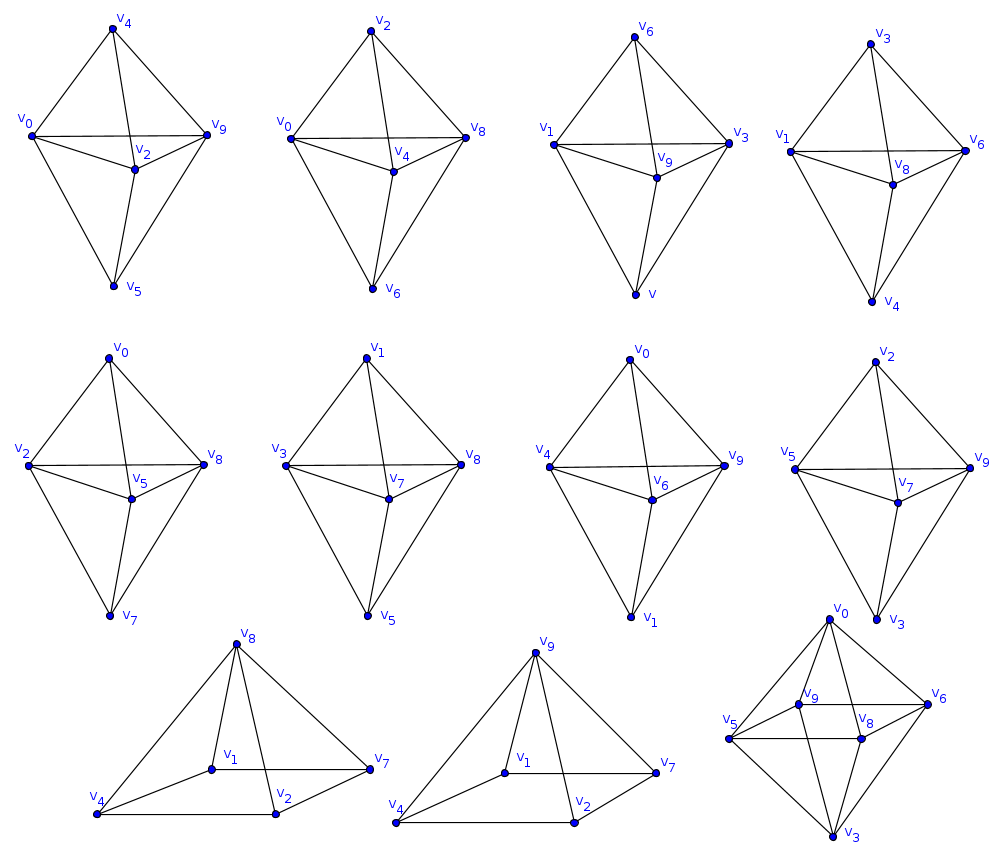}
\caption{These are the facets of the sphere $(10_{32,33}^0)$ from $F_0$ (top left) to $F_{10}$ (bottom right).}
\label{fig:10_32_33_11_a_facets}
\end{figure}

With reference to facet $F_0$, we may choose $\chi(v_{\hz0}, v_{\hz2}, v_{\hz4}, v_{\hz9}, v_{\hz i}) = \HM1$, 
for all $v_i\not\in F_0$. With this we can derive:
\begin{eqnarray}
\chi(v_{3}, v_{5}, v_{6}, v_{8}, v_{9}) \stackrel{F_{10}}{=} \HM0 , \label{eq:11_32_33_a_15}\\
\chi(v_{0}, v_{1}, v_{2}, v_{4}, v_{9}) = -1 & \stackrel{F_{\hz9}}{\Rightarrow} \chi(v_{1}, v_{2}, v_{4}, v_{5}, v_{9}) = \HM1 & \stackrel{F_{\hz0}}{\Rightarrow} \chi(v_{2}, v_{4}, v_{5}, v_{8}, v_{9}) = -1 , \hspace*{26pt}\label{eq:11_32_33_a_1}
 \end{eqnarray}
\begin{eqnarray} 
\chi(v_{0}, v_{2}, v_{4}, v_{8}, v_{9}) = -1 & \stackrel{F_{\hz1}}{\Rightarrow} \chi(v_{0}, v_{1}, v_{2}, v_{4}, v_{8}) = \HM1 & \stackrel{F_{\hz8}}{\Rightarrow} \chi(v_{1}, v_{2}, v_{4}, v_{6}, v_{8}) = -1  \hspace*{28pt}\nonumber\\ 
 & \stackrel{F_{\hz1}}{\Rightarrow} \chi(v_{2}, v_{4}, v_{6}, v_{8}, v_{9}) = -1 , \label{eq:11_32_33_a_3}
 \end{eqnarray}
\begin{eqnarray}
\chi(v_{0}, v_{2}, v_{4}, v_{8}, v_{9}) = -1 & \stackrel{F_{\hz1}}{\Rightarrow} \chi(v_{0}, v_{2}, v_{4}, v_{7}, v_{8}) = \HM1 & \stackrel{F_{\hz8}}{\Rightarrow} \chi(v_{2}, v_{4}, v_{7}, v_{8}, v_{9}) = \HM1 , \hspace*{26pt}\label{eq:11_32_33_a_5}
 \end{eqnarray}
\begin{eqnarray}
\chi(v_{0}, v_{2}, v_{4}, v_{8}, v_{9}) = -1 & \stackrel{F_{\hz1}}{\Rightarrow} \chi(v_{0}, v_{2}, v_{4}, v_{5}, v_{8}) = \HM1 & \stackrel{F_{\hz4}}{\Rightarrow} \chi(v_{0}, v_{2}, v_{5}, v_{6}, v_{8}) = -1 \nonumber\\ 
 & \stackrel{F_{\hz1}}{\Rightarrow} \chi(v_{0}, v_{2}, v_{3}, v_{6}, v_{8}) = -1 & \stackrel{F_{10}}{\Rightarrow} \chi(v_{0}, v_{1}, v_{3}, v_{6}, v_{8}) = -1 \nonumber\\ 
 & \stackrel{F_{\hz3}}{\Rightarrow} \chi(v_{1}, v_{3}, v_{6}, v_{8}, v_{9}) = -1 & \stackrel{F_{10}}{\Rightarrow} \chi(v_{2}, v_{3}, v_{6}, v_{8}, v_{9}) = -1,  \hspace*{26pt}\label{eq:11_32_33_a_7}
 \end{eqnarray}
\begin{eqnarray} 
\chi(v_{0}, v_{2}, v_{4}, v_{6}, v_{9}) = -1 & \stackrel{F_{\hz1}}{\Rightarrow} \chi(v_{0}, v_{1}, v_{2}, v_{4}, v_{6}) = \HM1 & \stackrel{F_{\hz6}}{\Rightarrow} \chi(v_{0}, v_{1}, v_{4}, v_{6}, v_{8}) = \HM1  \hspace*{28pt}\nonumber\\ 
 & \stackrel{F_{\hz3}}{\Rightarrow} \chi(v_{1}, v_{4}, v_{6}, v_{8}, v_{9}) = \HM1 , \label{eq:11_32_33_a_8}
 \end{eqnarray}
\begin{eqnarray}
\chi(v_{0}, v_{2}, v_{4}, v_{5}, v_{8}) = \HM1 & \stackrel{F_{\hz4}}{\Rightarrow} \chi(v_{0}, v_{2}, v_{3}, v_{5}, v_{8}) = \HM1 & \stackrel{F_{10}}{\Rightarrow} \chi(v_{0}, v_{3}, v_{5}, v_{7}, v_{8}) = \HM1 \nonumber\\ 
 & \stackrel{F_{\hz4}}{\Rightarrow} \chi(v_{0}, v_{1}, v_{5}, v_{7}, v_{8}) = \HM1 & \stackrel{F_{\hz5}}{\Rightarrow} \chi(v_{1}, v_{5}, v_{7}, v_{8}, v_{9}) = \HM1 , \hspace*{26pt}\label{eq:11_32_33_a_17}
 \end{eqnarray}
\begin{eqnarray} 
\chi(v_{0}, v_{1}, v_{2}, v_{4}, v_{6}) \stackrel{(\ref{eq:11_32_33_a_8})}{=} \HM1 & \stackrel{F_{\hz6}}{\Rightarrow} \chi(v_{0}, v_{1}, v_{3}, v_{4}, v_{6}) = \HM1 & \stackrel{F_{\hz3}}{\Rightarrow} \chi(v_{1}, v_{3}, v_{4}, v_{6}, v_{9}) = \HM1 \nonumber\\
& \stackrel{F_{\hz2}}{\Rightarrow} \chi(v_{0}, v_{1}, v_{3}, v_{6}, v_{9}) = \HM1 & \stackrel{F_{10}}{\Rightarrow} \chi(v_{0}, v_{3}, v_{6}, v_{7}, v_{9}) = \HM1 \hspace*{30pt}\nonumber\\ 
 & \stackrel{F_{\hz2}}{\Rightarrow} \chi(v_{3}, v_{6}, v_{7}, v_{8}, v_{9}) = -1 , \label{eq:11_32_33_a_9}
 \end{eqnarray}
\begin{eqnarray}
\chi(v_{0}, v_{1}, v_{3}, v_{6}, v_{9}) \stackrel{(\ref{eq:11_32_33_a_9})}{=} \HM1 & \stackrel{F_{\hz6}}{\Rightarrow} \chi(v_{0}, v_{1}, v_{6}, v_{7}, v_{9}) = -1 & \stackrel{F_{\hz2}}{\Rightarrow} \chi(v_{1}, v_{6}, v_{7}, v_{8}, v_{9}) = \HM1 , \hspace*{26pt}\label{eq:11_32_33_a_10}
 \end{eqnarray}
\begin{eqnarray}
\chi(v_{0}, v_{2}, v_{3}, v_{6}, v_{8}) \stackrel{(\ref{eq:11_32_33_a_7})}{=} -1 & \stackrel{F_{10}}{\Rightarrow} \chi(v_{0}, v_{3}, v_{4}, v_{6}, v_{8}) = \HM1 & \stackrel{F_{\hz3}}{\Rightarrow} \chi(v_{3}, v_{4}, v_{6}, v_{8}, v_{9}) = \HM1 , \hspace*{26pt}\label{eq:11_32_33_a_11}
 \end{eqnarray}
\begin{eqnarray}
\chi(v_{0}, v_{2}, v_{4}, v_{7}, v_{8}) \stackrel{(\ref{eq:11_32_33_a_5})}{=} \HM1 & \stackrel{F_{\hz4}}{\Rightarrow} \chi(v_{0}, v_{1}, v_{2}, v_{7}, v_{8}) = -1 & \stackrel{F_{\hz8}}{\Rightarrow} \chi(v_{1}, v_{2}, v_{5}, v_{7}, v_{8}) = -1 \hspace*{28pt}\nonumber\\ 
 & \stackrel{F_{\hz4}}{\Rightarrow} \chi(v_{2}, v_{5}, v_{7}, v_{8}, v_{9}) = -1 , \label{eq:11_32_33_a_14}
 \end{eqnarray}
\begin{eqnarray} 
\chi(v_{0}, v_{1}, v_{2}, v_{7}, v_{8}) \stackrel{(\ref{eq:11_32_33_a_14})}{=} -1 & \stackrel{F_{\hz8}}{\Rightarrow} \chi(v_{1}, v_{2}, v_{3}, v_{7}, v_{8}) = -1 & \stackrel{F_{\hz5}}{\Rightarrow} \chi(v_{1}, v_{3}, v_{7}, v_{8}, v_{9}) = \HM1 , \hspace*{26pt}\label{eq:11_32_33_a_16}\\
\chi(v_{0}, v_{3}, v_{5}, v_{7}, v_{8}) \stackrel{(\ref{eq:11_32_33_a_17})}{=} \HM1 & \stackrel{F_{\hz5}}{\Rightarrow} \chi(v_{3}, v_{5}, v_{7}, v_{8}, v_{9}) = \HM1 , \label{eq:11_32_33_a_18}
\end{eqnarray}

With these values for the partial chirotope, we can find some new values of $\chi$ using the Grassmann--Plücker relations:
\begin{eqnarray}
\{\chi(v_{7}, v_{8}, v_{9}, v_{1}, v_{3})\chi(v_{7}, v_{8}, v_{9}, v_{5}, v_{6}),&\chi(v_{7}, v_{8}, v_{9}, v_{1}, v_{5})\chi(v_{7}, v_{8}, v_{9}, v_{3}, v_{6}),&\nonumber\\ 
&\chi(v_{7}, v_{8}, v_{9}, v_{1}, v_{6})\chi(v_{7}, v_{8}, v_{9}, v_{3}, v_{5})\}&\nonumber\\ \stackrel{(\ref{eq:11_32_33_a_16}),(\ref{eq:11_32_33_a_17}),(\ref{eq:11_32_33_a_9}),(\ref{eq:11_32_33_a_10}),(\ref{eq:11_32_33_a_18})}{=}\{1\cdot \chi(v_{7}, v_{8}, v_{9}, v_{5}, v_{6}),& -1\cdot (-1),&\nonumber\\ 
&1\cdot 1\},\nonumber\\ 
 &\Rightarrow \chi(v_{7}, v_{8}, v_{9}, v_{5}, v_{6}) = -1 ,& \label{eq:11_32_33_a_12}
 \end{eqnarray}
\begin{eqnarray}
\{\chi(v_{6}, v_{8}, v_{9}, v_{2}, v_{3})\chi(v_{6}, v_{8}, v_{9}, v_{5}, v_{7}),&\chi(v_{6}, v_{8}, v_{9}, v_{2}, v_{5})\chi(v_{6}, v_{8}, v_{9}, v_{3}, v_{7}),&\nonumber\\ 
&\chi(v_{6}, v_{8}, v_{9}, v_{2}, v_{7})\chi(v_{6}, v_{8}, v_{9}, v_{3}, v_{5})\}&\nonumber\\ \stackrel{(\ref{eq:11_32_33_a_7}),(\ref{eq:11_32_33_a_12}),(\ref{eq:11_32_33_a_9}),(\ref{eq:11_32_33_a_15})}{=}\{(-1)\cdot 1,& -\chi(v_{6}, v_{8}, v_{9}, v_{2}, v_{5})\cdot 1,&\nonumber\\ 
& 0\},\nonumber\\ 
 &\Rightarrow \chi(v_{6}, v_{8}, v_{9}, v_{2}, v_{5}) = -1 ,& \label{eq:11_32_33_a_13}
 \end{eqnarray}
\begin{eqnarray}
\{\chi(v_{6}, v_{8}, v_{9}, v_{1}, v_{3})\chi(v_{6}, v_{8}, v_{9}, v_{4}, v_{7}),&\chi(v_{6}, v_{8}, v_{9}, v_{1}, v_{4})\chi(v_{6}, v_{8}, v_{9}, v_{3}, v_{7}),&\nonumber\\ 
&\chi(v_{6}, v_{8}, v_{9}, v_{1}, v_{7})\chi(v_{6}, v_{8}, v_{9}, v_{3}, v_{4})\}&\nonumber\\ \stackrel{(\ref{eq:11_32_33_a_7}),(\ref{eq:11_32_33_a_8}),(\ref{eq:11_32_33_a_9}),(\ref{eq:11_32_33_a_10}),(\ref{eq:11_32_33_a_11})}{=}\{(-1)\cdot \chi(v_{6}, v_{8}, v_{9}, v_{4}, v_{7}),& -1\cdot (-1),&\nonumber\\ 
&(-1)\cdot 1\},\nonumber\\ 
 &\Rightarrow \chi(v_{6}, v_{8}, v_{9}, v_{4}, v_{7}) = -1 ,& \label{eq:11_32_33_a_2}
 \end{eqnarray}
\begin{eqnarray}
\{\chi(v_{6}, v_{8}, v_{9}, v_{3}, v_{4})\chi(v_{6}v_{8}, v_{9}, v_{5}, v_{7}),&\chi(v_{6}, v_{8}, v_{9}, v_{3}, v_{5})\chi(v_{6}, v_{8}, v_{9}, v_{4}, v_{7}),&\nonumber\\ 
&\chi(v_{6}, v_{8}, v_{9}, v_{3}, v_{7})\chi(v_{6}, v_{8}, v_{9}, v_{4}, v_{5})\}&\nonumber\\ \stackrel{(\ref{eq:11_32_33_a_11}),(\ref{eq:11_32_33_a_12}),(\ref{eq:11_32_33_a_15}),(\ref{eq:11_32_33_a_9})}{=}\{1\cdot 1,& 0,&\nonumber\\ 
&1\cdot \chi(v_{6}, v_{8}, v_{9}, v_{4}, v_{5})\},\nonumber\\ 
 &\Rightarrow \chi(v_{6}, v_{8}, v_{9}, v_{4}, v_{5}) = -1 ,& \label{eq:11_32_33_a_6}
 \end{eqnarray}
\begin{eqnarray}
\{\chi(v_{5}, v_{8}, v_{9}, v_{2}, v_{4})\chi(v_{5}, v_{8}, v_{9}, v_{6}, v_{7}),&\chi(v_{5}, v_{8}, v_{9}, v_{2}, v_{6})\chi(v_{5}, v_{8}, v_{9}, v_{4}, v_{7}),&\nonumber\\ 
&\chi(v_{5}, v_{8}, v_{9}, v_{2}, v_{7})\chi(v_{5}, v_{8}, v_{9}, v_{4}, v_{6})\}&\nonumber\\ \stackrel{(\ref{eq:11_32_33_a_1}),(\ref{eq:11_32_33_a_12}),(\ref{eq:11_32_33_a_13}),(\ref{eq:11_32_33_a_14}),(\ref{eq:11_32_33_a_6})}{=}\{(-1)\cdot (-1),& -1\cdot \chi(v_{5}, v_{8}, v_{9}, v_{4}, v_{7}),&\nonumber\\ 
&(-1)\cdot (-1)\},\nonumber\\ 
 &\Rightarrow \chi(v_{5}, v_{8}, v_{9}, v_{4}, v_{7}) = \HM1 ,& \label{eq:11_32_33_a_4}
 \end{eqnarray}
 
Finally, we get the Grassmann--Plücker relation
\begin{eqnarray}
\{\chi(v_{4}, v_{8}, v_{9}, v_{2}, v_{5})\chi(v_{4}, v_{8}, v_{9}, v_{6}, v_{7}),&\chi(v_{4}, v_{8}, v_{9}, v_{2}, v_{6})\chi(v_{4}, v_{8}, v_{9}, v_{5}, v_{7}),&\nonumber\\ 
&\chi(v_{4}, v_{8}, v_{9}, v_{2}, v_{7})\chi(v_{4}, v_{8}, v_{9}, v_{5}, v_{6})\}&\nonumber\\ \stackrel{(\ref{eq:11_32_33_a_1}),(\ref{eq:11_32_33_a_2}),(\ref{eq:11_32_33_a_3}),(\ref{eq:11_32_33_a_4}),(\ref{eq:11_32_33_a_5}),(\ref{eq:11_32_33_a_6})}{=}\{1\cdot 1,& -1\cdot (-1),&\nonumber\\ 
&(-1)\cdot (-1)\},\label{eq:11_32_33_a_0}
 \end{eqnarray}
 which is neither $\{0\}$, nor contains $\{-1,1\}$. Thus, the the Grassmann--Plücker relations cannot be satisfied, 
so the sphere $(10_{32,33}^0)$ does not support an oriented matroid. In particular, it is not polytopal.
\end{proof}


\begin{example}\label{ex:10_33_35_12}
	There is exactly one $3$-sphere with $f$-vector $(10,33,35,12)$.
	This sphere $(10_{33,35})$ is given by the facet list
			\smallskip
			
			\begin{minipage}{0.49\textwidth}
			$F_{\hz0} = \{v_{\hz1}, v_{\hz4}, v_{\hz7}, v_{\hz9}\}$\\ 
			$F_{\hz1} = \{v_{\hz2}, v_{\hz4}, v_{\hz7}, v_{\hz9}\}$\\ 
			$F_{\hz2} = \{v_{\hz0}, v_{\hz2}, v_{\hz4}, v_{\hz5}, v_{\hz8}\}$\\ 
			$F_{\hz3} = \{v_{\hz0}, v_{\hz2}, v_{\hz4}, v_{\hz6}, v_{\hz9}\}$\\ 
			$F_{\hz4} = \{v_{\hz1}, v_{\hz3}, v_{\hz6}, v_{\hz7}, v_{\hz8}\}$\\ 
			$F_{\hz5} = \{v_{\hz1}, v_{\hz3}, v_{\hz4}, v_{\hz6}, v_{\hz9}\}$
			\end{minipage}
			\begin{minipage}{0.49\textwidth}
			$F_{\hz6} = \{v_{\hz0}, v_{\hz2}, v_{\hz5}, v_{\hz7}, v_{\hz9}\}$\\ 
			$F_{\hz7} = \{v_{\hz1}, v_{\hz3}, v_{\hz5}, v_{\hz7}, v_{\hz9}\}$\\ 
			$F_{\hz8} = \{v_{\hz0}, v_{\hz1}, v_{\hz4}, v_{\hz6}, v_{\hz8}\}$\\ 
			$F_{\hz9} = \{v_{\hz1}, v_{\hz2}, v_{\hz4}, v_{\hz7}, v_{\hz8}\}$\\ 
			$F_{10} = \{v_{\hz2}, v_{\hz3}, v_{\hz5}, v_{\hz7}, v_{\hz8}\}$\\ 
			$F_{11} = \{v_{\hz0}, v_{\hz3}, v_{\hz5}, v_{\hz6}, v_{\hz8}, v_{\hz9}\}$
			\end{minipage}
			\smallskip
			
			It is not polytopal. 
			It has a diagram based on each of the facets $F_2$, $F_3$, $F_4$, $F_5$, $F_6$, $F_7$, $F_8$, and $F_{10}$, 
			but not based on one of $F_0$, $F_1$, $F_9$, or $F_{11}$. 
			A diagram based on facet $F_2$ is given in Figure~\ref{fig:diagram_10_33_35_12}.
			The sphere cannot be represented by a fan.
\end{example}			
 
\begin{figure}
\centering 
\begin{minipage}{0.3\textwidth}
\vspace*{6pt}
$\underline{F_2}$\\
$v_{\hz0}  =  (1306, 2451, 4264)$ \\ 
$v_{\hz1}  =  (2471, 990, 1976)$ \\ 
$v_{\hz2}  =  (2881, 3713, 856)$ \\ 
$v_{\hz3}  =  (1412, 2367, 1947)$ \\ 
$v_{\hz4}  =  (2812, 766, 2282)$ \\ 
$v_{\hz5}  =  (1451, 2517, 2110)$ \\ 
$v_{\hz6}  =  (1965, 1505, 2347)$ \\ 
$v_{\hz7}  =  (1772, 2235, 976)$ \\ 
$v_{\hz8}  =  (636, 941, 864)$ \\ 
$v_{\hz9}  =  (1612, 2283, 2145)$
\end{minipage}
\begin{minipage}{0.49\textwidth}
\includegraphics[scale=0.75]{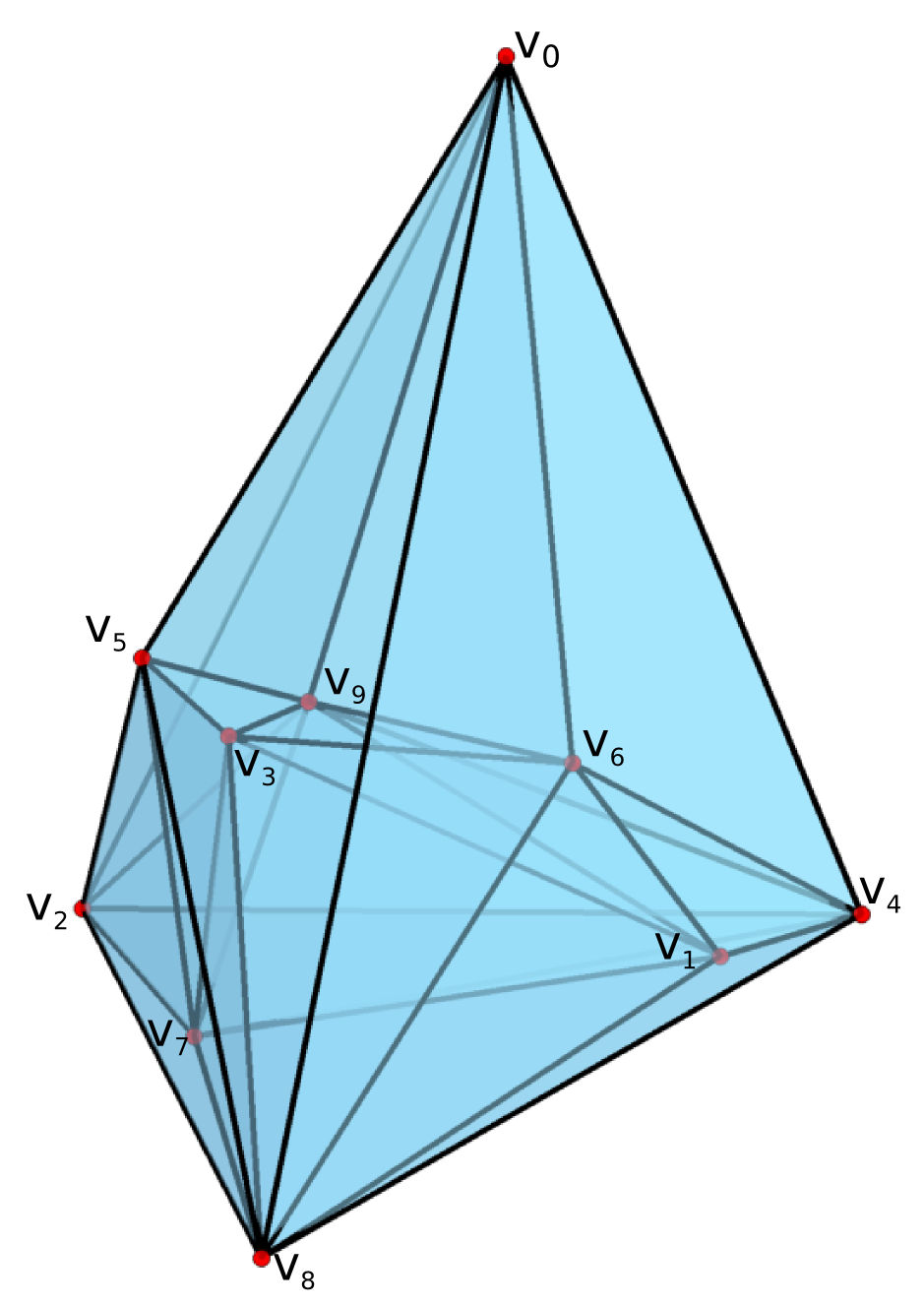}
\end{minipage}
\caption{A diagram based on facet $F_2$ for the sphere $(10_{33,35})$ with $f$-vector $(10,33,35,12)$.}
\label{fig:diagram_10_33_35_12} 
\end{figure}

\begin{examples}\label{ex:11_35_35_11}
There are exactly two $3$-spheres with $f$-vector $(11,35,35,11)$.
They are dual to each other.
These spheres $(11_{35}^0)$ and $(11_{35}^1)$ are given by facet lists
\smallskip
			
\begin{minipage}{0.49\textwidth}
$(11_{35}^0)$\\
$F_{\hz0} = \{v_{\hz1}, v_{\hz2}, v_{\hz4}, v_{\hz6}, v_{\hz9}\}$\\
$F_{\hz1} = \{v_{\hz3}, v_{\hz5}, v_{\hz7}, v_{\hz8}, v_{\hz9}\}$\\
$F_{\hz2} = \{v_{\hz0}, v_{\hz6}, v_{\hz7}, v_{\hz8}, v_{10}\}$\\
$F_{\hz3} = \{v_{\hz2}, v_{\hz3}, v_{\hz4}, v_{\hz8}, v_{\hz9}\}$\\
$F_{\hz4} = \{v_{\hz0}, v_{\hz1}, v_{\hz2}, v_{\hz5}, v_{\hz6}, v_{10}\}$\\
$F_{\hz5} = \{v_{\hz3}, v_{\hz4}, v_{\hz7}, v_{\hz8}, v_{10}\}$\\
$F_{\hz6} = \{v_{\hz1}, v_{\hz5}, v_{\hz6}, v_{\hz7}, v_{\hz9}\}$\\
$F_{\hz7} = \{v_{\hz0}, v_{\hz1}, v_{\hz2}, v_{\hz4}, v_{\hz8}, v_{10}\}$\\
$F_{\hz8} = \{v_{\hz3}, v_{\hz5}, v_{\hz6}, v_{\hz7}, v_{10}\}$\\
$F_{\hz9} = \{v_{\hz0}, v_{\hz2}, v_{\hz6}, v_{\hz7}, v_{\hz8}, v_{\hz9}\}$\\
$F_{10} = \{v_{\hz1}, v_{\hz3}, v_{\hz4}, v_{\hz5}, v_{\hz9}, v_{10}\}$
\end{minipage}
\begin{minipage}{0.49\textwidth}
$(11_{35}^1)$\\
$F_{\hz0}  =  \{v_{\hz2}, v_{\hz4}, v_{\hz7}, v_{\hz9}\}$\\ 
$F_{\hz1}  =  \{v_{\hz0}, v_{\hz4}, v_{\hz6}, v_{\hz7}, v_{10}\}$\\ 
$F_{\hz2}  =  \{v_{\hz0}, v_{\hz3}, v_{\hz4}, v_{\hz7}, v_{\hz9}\}$\\ 
$F_{\hz3}  =  \{v_{\hz1}, v_{\hz3}, v_{\hz5}, v_{\hz8}, v_{10}\}$\\ 
$F_{\hz4}  =  \{v_{\hz0}, v_{\hz3}, v_{\hz5}, v_{\hz7}, v_{10}\}$\\ 
$F_{\hz5}  =  \{v_{\hz1}, v_{\hz4}, v_{\hz6}, v_{\hz8}, v_{10}\}$\\ 
$F_{\hz6}  =  \{v_{\hz0}, v_{\hz2}, v_{\hz4}, v_{\hz6}, v_{\hz8}, v_{\hz9}\}$\\ 
$F_{\hz7}  =  \{v_{\hz1}, v_{\hz2}, v_{\hz5}, v_{\hz6}, v_{\hz8}, v_{\hz9}\}$\\ 
$F_{\hz8}  =  \{v_{\hz1}, v_{\hz2}, v_{\hz3}, v_{\hz5}, v_{\hz7}, v_{\hz9}\}$\\ 
$F_{\hz9}  =  \{v_{\hz0}, v_{\hz1}, v_{\hz3}, v_{\hz6}, v_{\hz9}, v_{10}\}$\\ 
$F_{10}  =  \{v_{\hz2}, v_{\hz4}, v_{\hz5}, v_{\hz7}, v_{\hz8}, v_{10}\}$
\end{minipage}
\smallskip
			
Both spheres are not fan-like, hence they have no star-shaped embedding. 
Furthermore, the sphere $(11_{35}^0)$ does not have a diagram with base $F_6$, $F_9$, or $F_{10}$;
the sphere $(11_{35}^1)$ has a diagram based on each of $F_4$ and $F_6$, 
but does not have a diagram with base $F_0$, $F_1$, $F_3$, $F_5$, $F_9$, or $F_{10}$.
A diagram for $(11_{35}^1)$ with base $F_6$ is given in Figure~\ref{fig:diagram_11_35_35_11_1}.
\end{examples}			

\begin{figure}[ht]
\centering 
\begin{minipage}{0.48\textwidth}
\vspace*{6pt}
$\underline{F_6}$\\
$v_{\hz0}  =  (0, 0, 0)$ \\ 
$v_{\hz1}  =  (1797, 1585, 512)$ \\ 
$v_{\hz2}  =  (2009, 2395, 1622)$ \\ 
$v_{\hz3}  =  (460, 1113, 648)$ \\ 
$v_{\hz4}  =  (0, 0, 1000)$ \\ 
$v_{\hz5}  =  (8565805/4137, 2055, 1316)$ \\ 
$v_{\hz6}  =  (2850, 426, 139)$ \\ 
$v_{\hz7}  =  (521, 1238, 853)$ \\ 
$v_{\hz8}  =  (2946124555/1064794, 1020, 770)$ \\ 
$v_{\hz9}  =  (423, 2580, 139)$ \\ 
$v_{10}  =  (1161, 1055, 677)$
\end{minipage}
\begin{minipage}{0.48\textwidth}
\includegraphics[scale=.8]{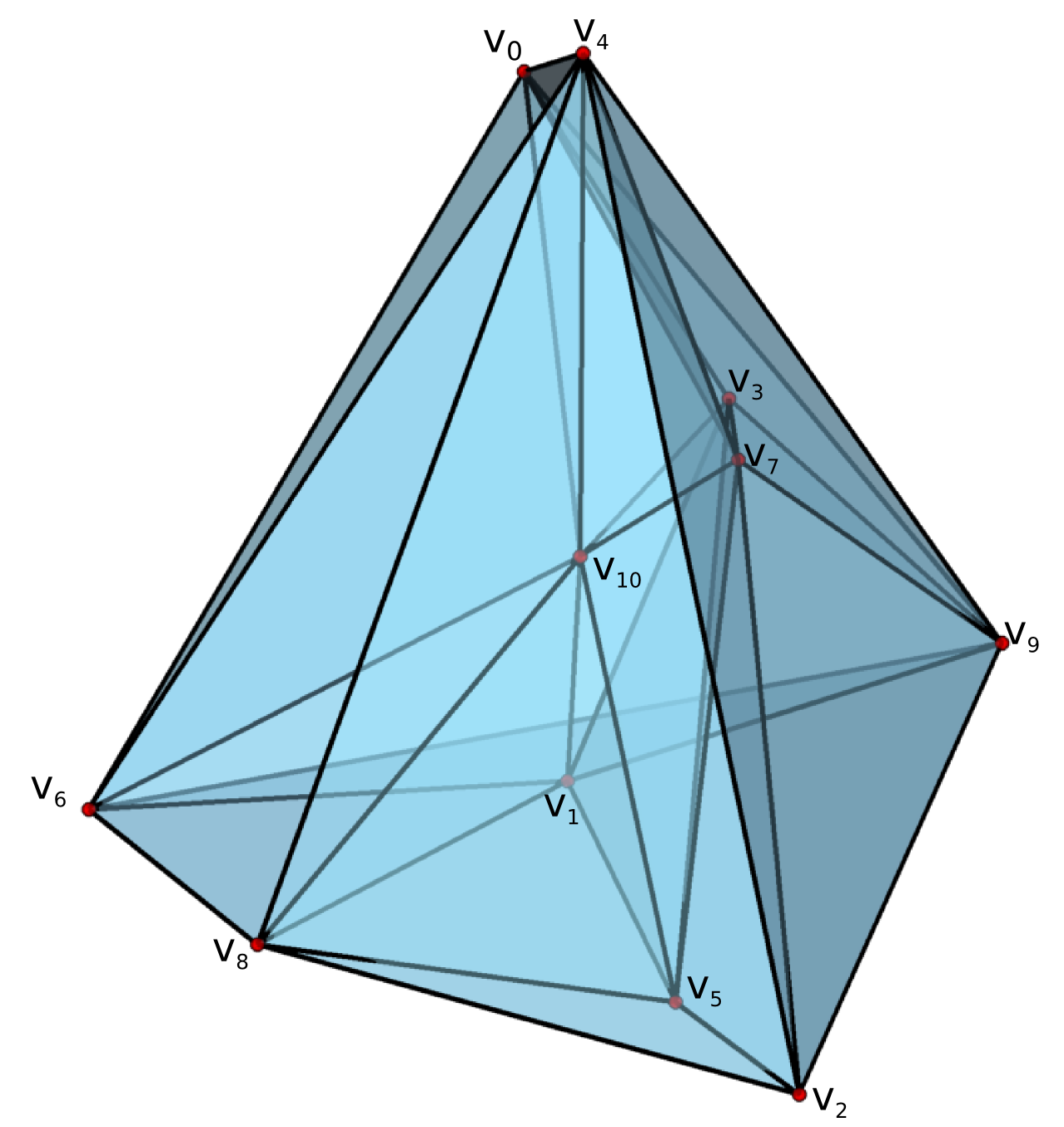}
\end{minipage}
\caption{A diagram based on facet $F_6$ for the sphere $(11_{35}^1)$ with $f$-vector $(11,35,35,11)$.} 
\label{fig:diagram_11_35_35_11_1}
\end{figure}

Similar details can be found in Brinkmann \cite[Sect.~3.2.4]{B_thesis} for the four self-dual $3$-spheres of Theorem~\ref{thm:12_vert}.

\section*{Acknowledgements}
We are very grateful to Marge Bayer and Moritz Firsching valuable comments and discussions.
The Algorithm in Section~\ref{sec:algorithm} was developed in very productive
exchanges with Katy Beeler, Hannah Sch{\"a}fer Sj{\"o}berg, and Moritz Schmitt.

\begin{small}

\end{small}

\end{document}